\documentclass[12pt]{amsart}
\usepackage{amsfonts}
\usepackage{amssymb}
\usepackage{bm}

 \newtheorem{theorem}{Theorem}[section]
 \newtheorem{cor}[theorem]{Corollary}
 
 \newtheorem{lemma}[theorem]{Lemma}
 
 \theoremstyle{definition}
 
 \theoremstyle{remark}

\numberwithin{equation}{section}
\DeclareMathOperator{\supp}{supp}

\renewcommand{\b}{\hm\cdot}

\begin{document}
\title[Extremal product-one free sequences and $|G|$-product-one free sequences]
{Extremal product-one free sequences and $|G|$-product-one free sequences of a metacyclic group}

\author[Y.K. Qu]{Yongke Qu}
\address{Department of Mathematics\\ Luoyang Normal University\\
LuoYang 471934, P.R.
CHINA}
\email{yongke1239@163.com}

\author[Y.L. Li]{Yuanlin Li*}
\address{Department of Mathematics and Statistics\\  Brock University\\ St. Catharines, ON L2S 3A1, Canada}
\email{yli@brock.ca}

\thanks{*Corresponding author: Yuanlin Li, E-mail: yli@brocku.ca}

\begin{abstract}
Let $G$ be a multiplicatively written finite group. We denote by $\mathsf E(G)$ the smallest integer $t$ such that every sequence of $t$ elements in $G$ contains a product-one subsequence of length $|G|$. In 1961, Erd\H{o}s, Ginzburg and Ziv proved that $\mathsf E(G)\leq 2|G|-1$ for every finite abelian group $G$ and this result is known as the Erd\H{o}s-Ginzburg-Ziv Theorem. In 2005, Zhuang and Gao conjectured that $\mathsf E(G)=\mathsf d(G)+|G|$ for every finite group, where $\mathsf d(G)$ is the small Davenport constant. Very recently, we confirmed this conjecture for the case when $G=\langle x, y| x^p=y^m=1, x^{-1}yx=y^r\rangle$ where $p$ is the smallest prime divisor of $|G|$ and $\mbox{gcd}(p(r-1), m)=1$. In this paper, we study the associated inverse problems on $\mathsf d(G)$ and $\mathsf E(G)$. Our main results characterize the structure of any product-one free sequence with extremal length $\mathsf d(G)$, and that of any $|G|$-product-one free sequence with extremal length $\mathsf E(G)-1$.
\end{abstract}

\date{}

\maketitle


\section{Introduction and main results}
Let $G$ be a finite group written multiplicatively and let $S=g_1\bm\cdot \ldots\bm\cdot g_{\ell}$ be a sequence over $G$ with length $\ell$. We use $$\pi(S)=\{g_{\tau (1)}\ldots g_{\tau (\ell)}: \tau \mbox { a permutation of } [1,\ell]\}\subseteq G$$ to denote the set of products of $S$. We say that $S$ is a {\sl product-one sequence} if $1\in\pi(S)$. We denote by $\mathsf d(G)$ the maximal integer $\ell$ such that there is a sequence of length $\ell$ over $G$ which has no nontrivial product-one subsequence. We denote by $\mathsf E(G)$ the smallest integer $t$ such that every sequence of $t$ elements in $G$ contains a product-one subsequence of length $|G|$. The problem of finding the precise value of the Davenport constant of a finite group and what is now known as the Erd\H{o}s-Ginzburg-Ziv Theorem (see below for details) have become the starting points of zero-sum theory. Since that time (dating back to the early 1960s), zero-sum theory has developed into a flourishing branch of additive and combinatorial number theory. In particular, two invariants $\mathsf d(G)$ and $\mathsf E(G)$ found wide attention for finite abelian groups (see \cite{Caro2,GG2006,GeK,GeR,Gr} for surveys). Although the main focus of $\mathsf d(G)$ and $\mathsf E(G)$ has been on abelian groups, research was never restricted to the abelian setting alone.

At the present time, two conjectures on $\mathsf E(G)$ can be found in the literature. The first one investigates the upper bound for $\mathsf E(G)$. In 1961, Erd\H{o}s, Ginzburg and Ziv \cite{EGZ1961} proved that $\mathsf E(G)\leq 2|G|-1$ for every finite abelian group $G$ and this result is now well known as the Erd\H{o}s-Ginzburg-Ziv Theorem. In 1976, Olson \cite{O} showed that $\mathsf E(G)\leq 2|G|-1$ holds for all finite groups. In 1984, Yuster and Peterson \cite{YP1984} showed that $\mathsf E(G)\leq 2|G|-2$ when $G$ is a non-cyclic solvable group. Later, Yuster \cite{YP1988} improved the result to $\mathsf E(G)\leq 2|G|-r$ provided that $|G|\geq 600((r-1)!)^2$. In 1996, Gao \cite{Gao1996} further improved the upper bound to $\mathsf E(G)\leq \frac{11|G|}{6}-1$. Recently, in 2010, Gao and Li \cite{GL2010} proved that $\mathsf E(G)\leq\frac{7|G|}{4}-1$ and  conjectured that $\mathsf E(G)\leq \frac{3|G|}{2}$ for all non-cyclic finite groups. We remark that $\frac{3|G|}{2}$ is the best possible upper bound. Most recently, Gao, Li, and Qu \cite{GLQ} proved that $\mathsf E(G)\leq \frac{3|G|}{2}$ holds for all non-cyclic groups of odd order.

The second conjecture proposed by Zhuang and Gao \cite{ZG2005} in 2005 describes a relationship between $\mathsf d(G)$ and $\mathsf E(G)$, which says $\mathsf E(G)=\mathsf d(G)+|G|$ holds for all finite groups. This relation has
been verified for all finite abelian groups (see \cite{Caro, G1996}) as well as for several types of non-abelian groups. These include: dihedral groups, dicyclic groups, $C_p\ltimes C_q$, $C_m\ltimes C_{mn}$ and all non-abelian nilpotent groups (see \cite{Bass2007, GL2008, Han2015, HZ2019,ZG2005} for details). Most recently in \cite{QL}, by computing the exact values of $\mathsf d(G)$ and $\mathsf E(G)$, we confirmed the above conjecture for a new class of metacyclic groups $G=C_p\ltimes C_m$, where $p$ is the smallest prime divisor of $|G|$ and $\mbox{gcd}(p(r-1), m)=1$.
\\

\noindent {\bf Theorem A.} \cite[Theorem 1.2]{QL} Let $G=\langle x, y| x^p=y^m=1, x^{-1}yx=y^r\rangle$, where $p$ is the smallest prime divisor of $|G|$ and $\mbox{gcd}(p(r-1), m)=1$. Then $\mathsf d(G)=m+p-2$ and $\mathsf E(G)=\mathsf d(G)+|G|=mp+m+p-2$.
\\

We remark that the investigation of this particular class of metacyclic groups is of special interest because it is closely related to solving the first conjecture. In fact, the main results and the key method developed in \cite{QL} have been used intensively  in \cite{GLQ} to show that $\mathsf E(G)\leq \frac{3|G|}{2}$ holds for all non-cyclic groups of odd order.

After solving the direct problem, which asks for the precise value of group invariants such as $\mathsf d(G)$ and $\mathsf E(G)$, it is natural to consider the associated inverse problem, which asks for determining the structure of extremal sequences. In particular, the associated inverse problem on $\mathsf d(G)$ (resp. $\mathsf E(G)$) seeks to determine the structure of product-one free (resp. $|G|$-product-one free) sequences of length $\mathsf d(G)$ (resp. $\mathsf E(G)-1$). Recently, Brochero Mart\'{\i}nez and Ribas solved the associated inverse problem on $\mathsf d(G)$ over dihedral groups, dicyclic groups, and a class of metacyclic groups $C_m\ltimes C_q$ (see \cite{BR2018,BR2019}). Oh and Zhong solved the associated inverse problem on $\mathsf E(G)$ over dihedral groups and dicyclic groups (see \cite{OZ}).
In this paper, we investigate the associated inverse problems on $\mathsf d(G)$ and $\mathsf E(G)$ over the above mentioned class of metacyclic groups $G=C_p\ltimes C_m$ and our main results are as follows.
\begin{theorem}\label{inversed(G)} Let $G=\langle x, y| x^p=y^m=1, x^{-1}yx=y^r\rangle$, where $p$ is the smallest prime divisor of $|G|$ and $\mbox{gcd}(p(r-1), m)=1$. Let $S$ be a sequence over $G$ with $|S|=\mathsf d(G)=m+p-2$. Then the following statements are equivalent.
\begin{enumerate}
\item[(a)] $S$ is product-one free.
\item[(b)] $S$ is of one of the following forms:
\begin{enumerate}
\item[(i)] If $G\not\cong C_2\ltimes C_3$ then $$S=(x^ay^{b_1})\bm\cdot\ldots\bm\cdot (x^ay^{b_{p-1}})\bm\cdot (y^c)^{[m-1]},$$ where $a\in[1,p-1]$, $b_i\in[0,m-1]$ for $i\in[1,p-1]$, $c\in[1,m-1]$ and $\mbox{gcd}(c,m)=1$.
\item[(ii)] If $G\cong C_2\ltimes C_3$ then $$S=x\bm\cdot xy\bm\cdot xy^2$$ or $$S= xy^{b}\bm\cdot (y^{c})^{[2]}$$ for $b\in [0,2]$ and $c\in[1,2]$.
\end{enumerate}
\end{enumerate}
\end{theorem}

\begin{theorem}\label{inverseE(G)} Let $G=\langle x, y| x^p=y^m=1, x^{-1}yx=y^r\rangle$, where $p$ is the smallest prime divisor of $|G|$ and $\mbox{gcd}(p(r-1), m)=1$. Let $S$ be a sequence over $G$ with $|S|=\mathsf E(G)-1=mp+m+p-3$. Then the following statements are equivalent.
\begin{enumerate}
\item[(a)] $S$ is $|G|$-product-one free.
\item[(b)] $S$ is of one of the following forms:
\begin{enumerate}
\item[(i)] If $G\not\cong C_2\ltimes C_3$, then $$S=(x^ay^{b_1})\bm\cdot\ldots\bm\cdot (x^ay^{b_{p-1}})\bm\cdot (y^{c_1})^{[k_1m-1]}\bm\cdot (y^{c_2})^{[k_2m-1]},$$ where $a\in[1,p-1]$, $b_i, c_1,c_2\in[0,m-1]$ for $i\in[1,p-1]$, $\mbox{gcd}(c_1-c_2,m)=1$ and $k_1+k_2=p+1$.
\item[(ii)] If $G\cong C_2\ltimes C_3$, then  $$S=x\bm\cdot xy \bm\cdot xy^{2}\bm\cdot  1^{[5]},$$ or $$S=(xy^{b})\bm\cdot (y^{c_1})^{[5]}\bm\cdot (y^{c_2})^{[2]},$$ where $b, c_1,c_2\in [0,2]$ with $\mbox{gcd}(c_1-c_2,3)=1$.
\end{enumerate}
\end{enumerate}
\end{theorem}

This paper is organized in the following way. In Section $2$, we describe some basic notation and preliminary assertions. In Section $3$, we collect some special properties for the above mentioned metacyclic groups. Our first main result is then proved in Section $4$, which characterizes the structure of an extremal product-one free sequence of length $\mathsf d(G)$. Finally, in Section $5$, we prove our second main result, which characterizes the structure of a extremal $|G|$-product-one free sequence of length $\mathsf E(G)-1$.

\section{Preliminaries}

We follow the notation and conventions detailed in \cite{GG2013}.

For $a,b\in \mathbb{R}$, we set $[a,b]=\{x\in \mathbb{Z}: a\leq x\leq b\}$. For integers $m,n\in \mathbb{Z}$, we denote by $\mbox{gcd}(m,n)$ the greatest common divisor of $m$ and $n$.

Let $G$ be a finite multiplicative group. If $A$ and $B$ are subsets of $G$, we define the product-set as $AB=\{ab:a\in A,b\in B\}$. If $A \subseteq G$ is a nonempty subset, then denote by  $\langle A \rangle $  the subgroup  of $G$ generated by $A$.  Recall that by ``a {\sl sequence} over a group $G$," we mean a finite, unordered sequence where repetition of elements is allowed. We view sequences over $G$ as elements of the free abelian monoid $\mathcal{F}(G)$, denote multiplication in $\mathcal{F}(G)$ by the bold dot symbol $\bm\cdot$ rather than by juxtaposition, and use brackets for all exponentiation in $\mathcal{F}(G)$.

A sequence $S \in \mathcal F(G)$ can be written in the form $S= g_1  \bm \cdot g_2 \bm \cdot \ldots \bm\cdot g_{\ell},$ where $|S|= \ell$ is the {\it length} of $S$. For $g \in G$, let $\mathsf v_g(S) = |\{ i\in [1, \ell] : g_i =g \}|\,  $ denote the {\it multiplicity} of $g$ in $S$. A sequence $T \in \mathcal F(G)$ is called a {\it subsequence } of $S$ and is denoted by $T \mid S$ if  $\mathsf v_g(T) \le \mathsf v_g(S)$ for all $g\in G$. Denote by $S \bm\cdot T^ {[-1]}$  the subsequence of $S$ obtained by removing the terms of $T$ from $S$. For every subset $A$ of $G$, let $S_A$ denote the subsequence of $S$ consisting of all terms of $S$ contained in $A$.

If $S_1, S_2 \in \mathcal F(G)$, then $S_1 \bm\cdot S_2 \in \mathcal F(G)$ denotes the sequence satisfying $\mathsf v_g(S_1 \bm\cdot S_2) = \mathsf v_g(S_1 ) + \mathsf v_g( S_2)$ for all $g \in G$. For convenience, we  write
\begin{center}
 $g^{[k]} = \underbrace{g \bm\cdot \ldots \bm\cdot g}_{k} \in \mathcal F(G)\quad$
\end{center}
for $g \in G$ and $k \in \mathbb{N}_0$. Suppose $S= g_1 \b g_2 \bm \cdot \ldots \bm\cdot g_{\ell} \in \mathcal F(G)$. Let
$$\pi (S) = \{g_{\tau(1)}\ldots g_{\tau(\ell)}: \tau \mbox{ a permutation of } [1, \ell] \} \subseteq G$$
denote the {\it set of products} of $S$. Let
$$\Pi_n(S) = \cup _{T\mid S,\ |T| = n}\pi(T)$$
denote the {\it set of all $n$-products} of $S$. Let
$$\Pi(S) = \cup_{1 \le n \le \ell}\Pi_n(S)$$
denote the {\it set of all subsequence products} of $S$. A sequence $S$ is called
\begin{itemize}
\item[$\bullet$]  {\it product-one} if $1 \in \pi(S)$;
\item[$\bullet$] {\it product-one free} if $1\not\in \Pi(S)$;
\item[$\bullet$] {\it $|G|$-product-one free} if $1\not\in \Pi_{|G|}(S)$;
\item[$\bullet$] {\it minimal product-one } if $1\in\pi(S)$ and $S$ has no proper nontrivial product-one subsequence.
\end{itemize}

We now collect some necessary lemmas.

Let $\mathbf{A}=(A_1,\ldots, A_{\ell})$ be a sequence of finite subsets of $G$ and $k\leq \ell$. We define $$\Pi^{k}(\mathbf{A})=\{a_{i_1}\ldots a_{i_{k}}:1\leq i_1<\cdots<i_{k}\leq \ell\mbox{ and } a_{i_j}\in A_{i_j} \mbox{ for every } 1\leq j\leq k\},$$
and
$$\Pi(\mathbf{A})=\cup_{1\leq k\leq \ell}\Pi^{k}(\mathbf{A}).$$
Let $A$ be a subset of $G$ and $\mbox{stab}(A)=\{g\in G: gA=A\}$ its stabilizer. The following is the well known Kneser's Addition Theorem, and we direct the interested reader to \cite{GeK, Na} for a detailed proof.

\begin{lemma}\label{Kneser} (Kneser) Let $A_1, \ldots, A_r$ be finite nonempty subsets of an abelian group, and let $H=\mbox{stab}(A_1\ldots A_r)$. Then,
$$
|A_1\ldots A_r|\geq |A_1H|+\cdots +|A_rH|-(r-1)|H|.
$$
\end{lemma}

The following lemma is a generalization of Kneser's Addition Theorem which is crucial for proving our results.

\begin{lemma}\cite[Theorem 1.3]{DGM2009}\cite[Theorem 13.1]{Gr}\label{genKneser}
Let $\mathbf{A}=(A_1,\ldots, A_{\ell})$ be a sequence of finite subsets of an abelian group $G$, let $k\leq \ell$, and $H=\mbox{stab}(\Pi^{k}(\mathbf{A}))$. If $\Pi^{k}(\mathbf{A})$ is nonempty, then $$|\Pi^{k}(\mathbf{A})|\geq |H|\bigg(1-k+\sum_{Q\in G/H}min\big\{k,|\{i\in[1,\ell]:A_i\cap Q\neq \emptyset\}|\big\}\bigg).$$
\end{lemma}

\begin{lemma}\cite[Theorem 2.1(a)]{GGS2007}\label{n-1}
Let $G=\langle g\rangle$ be cyclic of order $n\geq 3$, and let $S$ be a product-one free sequence of length $|S|=n-1$. Then $S=(g^k)^{[n-1]}$ for some $k$ with $\mbox{gcd}(k, n)=1$.
\end{lemma}

\begin{lemma}\label{normalsequence}
Let $G$ be a finite group, $S$ be a sequence over $G$ and $T\mid S$ be a product-one subsequence with maximal length. Let $W=S\bm\cdot T^{[-1]}$. If $|T|=|S|-\mathsf d(G)$, then $\mathsf v_g(W)\geq 1$ for any $g\in G\setminus \{1\}$ with $\mathsf v_g(S)\geq 1$.
\end{lemma}
The proof of this result for abelian groups can be found in \cite[Theorem 1.2]{GZ2006}, and the same proof works for all finite groups.

\begin{lemma}\cite[Lemma 2.4]{HZ2019}\label{knzerosumfree}
Let $G=\langle g\rangle$ be a cyclic group of order $n$ and let $S$ be a sequence over $G$.
\begin{enumerate}
\item[(i)] If $|S|=kn+n-1$, then $S$ contains a product-one subsequence $T$ of length $kn$;
\item[(ii)] If $|S|=kn+n-2$ and $S$ contains no product-one subsequence of length $kn$, then $S$ must be of the type $S=(g^{a})^{[k_1n-1]}\bm\cdot (g^b)^{[k_2n-1]}$, where $k_1+k_2=k+1$ and $\mbox{gcd}(a-b, n)=1$. Moreover, $\Pi_{kn-2}(S)=G$.
\end{enumerate}
\end{lemma}

\begin{lemma}\cite[Lemma 2.2]{Na}\label{A+B>G}
Let $G$ be a finite abelian group, and let $A$ and $B$ be subsets of $G$ such that $|A|+|B|>|G|$. Then $AB=G$.
\end{lemma}

\begin{lemma}\cite[Lemma 4]{OW}\label{Kemperman}
Suppose $A$ and $B$ are finite subsets of an arbitrary group and $1\in A\cap B$. If $ab=1$ (for $a\in A, b\in B$) has no solution except $a=b=1$, then $|AB|\geq |A|+|B|-1$.
\end{lemma}

\begin{lemma}\label{productonefree}
Let $\mathbf{A}=(A_1,A_2,\ldots,A_{\ell})$ be a sequence of subsets of a group $G$. If $1\notin \Pi(\mathbf{A})$, then $|\Pi(\mathbf{A})|\geq \sum_{i=1}^{\ell} |A_i|$.
\end{lemma}
\begin{proof}
We proceed by induction on $\ell$. If $\ell=1$, then the result holds trivially. Assume that the result is true for $\ell-1 (\geq 1)$. Next, we prove that the result is also true for $\ell$. Since $1\notin \Pi(\mathbf{A})$, we have $1\notin\Pi(A_1,A_2,\ldots,A_{\ell-1})$ and $1\notin A_{\ell}$. Let $$A'=\Pi(A_1,A_2,\ldots,A_{\ell-1})\cup \{1\}\mbox{ \ \ \ and \ \ \ }A_{\ell}'=A_{\ell}\cup\{1\}.$$
By the induction hypothesis, $|\Pi(A_1,A_2,\ldots,A_{\ell-1})|\geq \sum_{i=1}^{\ell-1} |A_i|$. Thus
$$|A'|\geq \sum_{i=1}^{\ell-1} |A_i|+1.$$
Since $1\notin \Pi(\mathbf{A})$, we conclude that $ab=1$ where $a\in A'$ and $b\in A_{\ell}'$, if and only if $a=b=1$. Clearly, $\Pi(\mathbf{A})=(A'A_{\ell}')\setminus \{1\}$. Therefore, by Lemma~\ref{Kemperman} we have
$$|\Pi(\mathbf{A})|=|A'A_{\ell}'|-1\geq |A'|+|A_{\ell}'|-2\geq \sum_{i=1}^{\ell} |A_i|.$$
This completes the proof.
\end{proof}

As a consequence of Lemma~\ref{productonefree}, we obtain the following result.

\begin{cor}\cite[Lemma 2.4]{GLP2014}\label{zerofree}
Let $S$ be a product-one free sequence over $G$. Then $|\Pi(S)|\geq |S|$.
\end{cor}
\section{Some Properties of the Group $G=\langle x, y| x^p=y^m=1, x^{-1}yx=y^r\rangle$}
From now on, let $G=\langle x, y| x^p=y^m=1, x^{-1}yx=y^r\rangle\cong C_p\ltimes C_m$, where $p$ is the smallest prime divisor of $|G|$, and $\mbox{gcd}(p(r-1), m)=1$. Let $K=\langle x\rangle$ and $N=\langle y\rangle$. Then $C_{m}\cong N\lhd G$ and $K\cong G/N\cong C_p$. Let $N_i=x^iN$ be the $i$th coset of $N$ in $G$ for $0\leq i\leq p-1$. Let $\varphi$ be the canonical homomorphism from $G$ onto $G/N$. Then for each sequence $T$ over $G$, $\varphi(T)$ is a sequence over $G/N$. Note that if $\varphi(T)$ is a product-one sequence over $G/N$, then $\pi(T)\subseteq N$ (because $\varphi(\pi(T))=\pi(\varphi(T))=1$). We first recall several useful lemmas whose proofs can be found in \cite{QL}.

\begin{lemma}\cite[Lemma 3.1]{QL}\label{basic} Let $G$ be as above. Let $M$ be any subgroup of $N=\langle y\rangle$, $u$ be an element of $N$ and $0\leq s<s'\leq p-1$. Then
\begin{itemize}
\item[(i)] If $u^{r^s}\in M$, then $\{u, u^{r},\ldots, u^{r^{p-1}}\}\subseteq M$.
\item[(ii)] If both $u^{r^s}$ and $u^{r^{s'}}$ are in the same coset of $M$, then $\{u, u^{r},\ldots, u^{r^{p-1}}\}\subseteq M$.
\item[(iii)] If $u\neq 1$, then $u^{r^s}\neq u^{r^{s'}}$.
\end{itemize}
\end{lemma}

\begin{lemma}\cite[Lemma 3.2]{QL}\label{conjugationone} Let $T=g_1\bm\cdot \ldots \bm\cdot g_t$ be a sequence over $G$ such that $\varphi (T)$ is a minimal product-one sequence over $G/N$. Then for any $u\in\pi(T)$, we have $\pi(T)\supseteq \{u^{r^{s_{1}}}, u^{r^{s_{2}}}, \ldots, u^{r^{s_{t}}}\}$ for some subset $\{s_1,\ldots, s_t\}\subseteq [0,p-1]$. Moreover, if $u\neq 1$, then $u^{r^{s_i}}\neq u^{r^{s_j}}$ for all $1\leq i<j\leq t\leq p$.
\end{lemma}

\begin{lemma}\cite[Lemma 3.3]{QL}\label{conjugationtwo}
Let $T_0$ be a sequence of $p$ elements in $N_i$ for some $i\in [1,p-1]$. For every $j\in [1,\ell]$, let $T_j$ be a sequence over $G$ such that $\pi(T_j)\cap N\neq \emptyset$, and let $u_j\in \pi(T_j)\cap N$. Then, for every $t\in[1,\ell]$, $\pi(T_0\bm\cdot T_1\bm\cdot\ldots\bm\cdot T_{t})$ contains the product set $\pi(T_0)\{u_1,u_1^{r},\ldots,u_1^{r^{p-1}}\}\ldots \{u_t,u_t^{r},\ldots,u_t^{r^{p-1}}\}$.
\end{lemma}

The next lemma provides some additional properties of $G$.
\begin{lemma}\label{mequiv1}
Let $G$ be as above. Then
\begin{itemize}
\item[(i)] $m\equiv1 \pmod p$ and $\mbox{gcd}(r^a-1,m)=1$ for every $p\nmid a$.
\item[(ii)] For every $1\neq h\in N$, the centralizer $C_G(h)\subseteq N$.
\item[(iii)] Every element $g\in G\setminus N$ has order $p$, i.e., $\mbox{ord}(g)=p$.
\end{itemize}
\end{lemma}

\begin{proof} (i) Proof of $m\equiv1 \pmod p$ can be found in \cite[Remark 3.4]{QL}. Now, we prove that $\mbox{gcd}(r^a-1,m)=1$ for every $p\nmid a$. Assume to the contrary that $\mbox{gcd}(r^{a}-1,m)\neq 1$, then there exists a prime divisor $q$ of $m$ such that $r^{a}-1\equiv 0 \pmod q$. Since $r^{p}\equiv1 \pmod m$, we have $r^{p}\equiv1 \pmod q$. Since $\mbox{gcd}(r-1,m)=1$, we have $r\not\equiv1 \pmod q$. Therefore, $r$ has order $p$ modulo $q$. Thus $p|a$, yielding a contradiction. Hence, we must have $\mbox{gcd}(r^{a}-1,m)=1$.

(ii) Let $h=y^{b_1}$ and $g=x^{a}y^{b_2}$ where $a\in [0,p-1]$, $b_1\in [1,m-1]$ and $b_2\in [0,m-1]$. Let $[h,g](\triangleq h^{-1}g^{-1}hg)$ be the commutator of $h$ and $g$. Then $[h,g]=[y^{b_1},x^{a}y^{b_2}]=[y^{b_1},x^{a}]=y^{b_1(r^{a}-1)}=h^{r^{a}-1}$. If $g\in C_G(h)$, then $[h,g]=1$. Thus, $h^{r^{a}-1}=1$. If $a\neq 0$, then by (i), $\mbox{gcd}(r^{a}-1,m)=1$ and thus $h^{r^{a}-1}\neq 1$, yielding a contradiction. Therefore, $a=0$ and thus $g\in N$.

(iii) Since $g\in G\setminus N$, without loss of generality we may assume that $g=x^{-a}h$ for some $a\in[1,p-1]$ and $h\in N$. Thus $g^{p(r-1)}=((x^{-a}hx^a)(x^{-2a}hx^{2a})\ldots (x^{-ap}hx^{ap}))^{r-1}=(h^{r^a+r^{2a}+\ldots +r^{a(p-1)}+1})^{r-1}=h^{r^{ap}-1}=1$. Since $(g^p)^m=1$ and $\mbox{gcd}(m,r-1)=1$, we get $g^p=1$, implying that $\mbox{ord}(g)=p$ as desired. This completes the proof of the lemma.
\end{proof}

The following lemma characterizes a minimal product-one sequence over $G\setminus N$.

\begin{lemma}\label{pit=1}
Let $T$ be a minimal product-one sequence over $G$ of length $p$ with elements in $G\setminus N$ such that $\pi(T)=\{1\}$. Then $T=g^{[p]}$ for some $g\in G\setminus N$.
\end{lemma}
\begin{proof}
Let $T=g_1\bm\cdot \ldots \bm\cdot g_p$ where $g_i\in G\setminus N$ for all $i\in [1,p]$. Since $g_1 \ldots  g_p\in \pi(T)$, $g_2 g_1 g_3 \ldots  g_p\in \pi(T)$ and $\pi(T)=\{1\}$, we obtain $g_1 \ldots  g_p=g_2 g_1  g_3 \ldots g_p$, whence $g_1 g_2= g_2 g_1$, so $[g_1, g_2]=1$. We next show that $g_2\in K_1=\langle g_1\rangle$. Since $g_1\notin N$, we have $G=K_1N$. Let $g_2=g_1^ah$ for some $a\in[1,p-1]$ and $h\in N$. Since $1=[g_1,g_2]=[g_1,g_1^{a}h]=[g_1,h]$, we get $g_1\in C_h(G)$. It follows from Lemma~\ref{mequiv1}~(ii) that $h=1$ (as $g_1\notin N$), so $g_2=g_1^a\in K_1$. Similarly, we obtain that $[g_1, g_i]=1$ for all $i\in [1,p]$, and thus $g_i\in K_1$ for all $i\in[1,p]$. Thus $T$ is a minimal product-one sequence of length $p$ over $K_1$. Since $K_1\cong C_p$, it follows from Lemma~\ref{n-1} that $T=g^{[p]}$ for some $g\in K_1\setminus \{1\}\subseteq G\setminus N$. This completes the proof.
\end{proof}
\section{Extremal Product-One Free Sequences}
In this section, we will characterize the structure of an extremal product-one free sequence of length $\mathsf d(G)$ and provide a proof for Theorem~\ref{inversed(G)}. We first prove the following crucial lemma.

\begin{lemma}\label{2p-1}
Let $T=g_1\bm\cdot\ldots \bm\cdot g_{\ell}$ be a sequence over $N_a=x^aN$ for some $a\in [1,p-1]$ with length $\ell=2p-1$. If $|\Pi_p(T)|= p$, then $p=2$ and there exists a subgroup $H$ of $N$ with $|H|=3$ such that $\Pi_p(T)=H\setminus \{1\}$.
\end{lemma}
\begin{proof}
Since $a\in [1,p-1]$ and $\mbox{ord}(x)=p$, we have $\langle x^a\rangle=\langle x\rangle$. By replacing $x^a$ with $x$ and $r^a$ with $r$ if needed, we may always assume that $a=1$. We first show that
\begin{align}\label{hT}
\mathsf h(T)\leq p-1.
\end{align}

Since $|\Pi_p(T)|= p$, there exists a subsequence $T'=g_{i_1}\bm\cdot\ldots \bm\cdot g_{i_p}$ of $T$ such that $1\neq u=g_{i_1}\ldots g_{i_p}\in \Pi_p(T)$. By Lemma~\ref{conjugationone}, $\{u,u^r,\ldots, u^{r^{p-1}}\}\subseteq\pi(T')\subseteq \Pi_p(T)$. Since $|\{u,u^r,\ldots, u^{r^{p-1}}\}|=p=|\Pi_p(T)|$, we have $\{u,u^r,\ldots, u^{r^{p-1}}\}=\Pi_p(T)$ and thus $1\notin \Pi_p(T)$. Note that if $\mathsf h(T)\geq p$, then $g^{[p]}\mid T$ for some $g\mid T$. Since $a\in [1,p-1]$, we have $g\notin N$. By Lemma~\ref{mequiv1}~(iii), $g^p=1$. Hence $1=g^p\in \Pi_p(T)$, yielding a contradiction. Therefore, we must have $\mathsf h(T)\leq p-1$.

Let $g_i=xh_i$ where $h_i\in N$ for all $i\in [1,\ell]$. Then
$$T=(xh_1)\bm\cdot\ldots \bm\cdot (xh_{\ell}).$$
Since $\mathsf h(T)\leq p-1$, we can obtain the following factorization of $T$:
$$T=A_1\bm\cdot \ldots \bm\cdot A_p$$
such that each subsequence $A_i$ is actually a subset for every $i\in [1,p]$, $|A_i|=2$ for every $i\in[1,p-1]$, and $|A_p|=1$. Let $\mathbf{A}=(A_1, A_2, \ldots, A_p)$. Without loss of generality, we may assume that $A_i=\{g_{2i-1}, g_{2i}\}$ for every $i\in [1,p-1]$ and $A_p=\{g_{2p-1}\}$.

Now, we construct a sequence of subsets $\mathbf{B}$ of $N$ such that
$$\Pi^p(\mathbf{B})\subseteq \Pi_p(T)\mbox{ with }\mbox{stab}(\Pi^p(\mathbf{B}))=\{1\}\mbox{ and }|\Pi^p(\mathbf{B})|=p.$$ Let $B_i=\{h_{2i-1}^{r^{p-i}}, h_{2i}^{r^{p-i}}\}$ for $i\in [1,p-1]$ and $B_p=\{h_{2p-1}\}$. Since $A_i$ is a subset, we have $g_{2i-1}=xh_{2i-1}\neq xh_{2i}=g_{2i}$ for all $i\in [1,p-1]$. So $h_{2i-1}\neq h_{2i}$ and thus $|B_i|=2$ for all $i\in [1,p-1]$. Let $\mathbf{B}=(B_1,\ldots, B_p)$ and $\mbox{stab}(\Pi^p(\mathbf{B}))=M$. A straight forward computation shows that
$$\Pi^p(\mathbf{B})=A_1\ldots A_{p-1}A_p \subseteq\Pi_p(T).$$
If $|M|\neq 1$, then $|M|> p$. By Lemma~\ref{Kneser}, $|\Pi_p(T)|\geq |\Pi^p(\mathbf{B})|\geq |M|> p$, yielding a contradiction.

Therefore, we must have that $|M|=1$ and thus $M=\{1\}$. By Lemma~\ref{Kneser}, $|\Pi^p(\mathbf{B})|\geq \sum_{i=1}^p|B_i|-p+1=p$. Since $|\Pi_p(T)|= p$ and $\Pi^p(\mathbf{B})\subseteq \Pi_p(T)$, we have $p=|\Pi_p(T)|\geq |\Pi^p(\mathbf{B})|\geq p$. Therefore, $|\Pi^p(\mathbf{B})|=p$.

Next, we prove the first part of the conclusion, i.e., $p=2$. Assume to the contrary that $p\geq 3$. Let $\mathbf{B_{1,2}}=(B_1,B_2)$ and $\mathbf{\overline{B_{1,2}}}=(B_3,\ldots, B_p)$. Since $\Pi^p(\mathbf{B})=\Pi^2(\mathbf{B_{1,2}})\Pi^{p-2}(\mathbf{\overline{B_{1,2}}})$ and $M=\{1\}$, we have $\mbox{stab}(\Pi^2(\mathbf{B_{1,2}}))=\mbox{stab}(\Pi^{p-2}(\mathbf{\overline{B_{1,2}}}))=\{1\}$. By Lemma~\ref{Kneser}, we have
$$|\Pi^2(\mathbf{B_{1,2}})|\geq |B_1|+|B_2|-1=3$$
and
$$|\Pi^{p-2}(\mathbf{\overline{B_{1,2}}})|\geq \sum_{k=3}^p|B_k|-((p-2)-1)=2(p-3)+1-(p-3)=p-2.$$

Note that $\Pi^p(\mathbf{B})=\Pi^2(\mathbf{B_{1,2}})\Pi^{p-2}(\mathbf{\overline{B_{1,2}}})$ and $\mbox{stab}(\Pi^p(\mathbf{B}))=M=\{1\}$. Since $|\Pi^p(\mathbf{B})|=p$, by Lemma~\ref{Kneser}, $$p=|\Pi^p(\mathbf{B})|\geq |\Pi^2(\mathbf{B_{1,2}})|+|\Pi^{p-2}(\mathbf{\overline{B_{1,2}}})|-1\geq 3+(p-2)-1=p.$$
Therefore, all the above equal signs hold and thus $|\Pi^2(\mathbf{B_{1,2}})|=3$. Note that
\begin{align*}
\Pi^2(\mathbf{B_{1,2}})=&\{h_1^{r^{p-1}}, h_{2}^{r^{p-1}}\}\{h_{3}^{r^{p-2}}, h_{4}^{r^{p-2}}\}\\
=&\{h_1^{r^{p-1}}h_{3}^{r^{p-2}}, h_1^{r^{p-1}}h_{4}^{r^{p-2}}, h_2^{r^{p-1}}h_{3}^{r^{p-2}}, h_2^{r^{p-1}}h_{4}^{r^{p-2}}\}.
\end{align*}
Since $h_1\neq h_2$, we have
$$h_1^{r^{p-1}}h_{3}^{r^{p-2}}\neq h_2^{r^{p-1}}h_{3}^{r^{p-2}},\ \ \ \  h_1^{r^{p-1}}h_{4}^{r^{p-2}}\neq h_2^{r^{p-1}}h_{4}^{r^{p-2}};$$
similarly, since $h_{3}\neq h_{4}$, we have
$$h_1^{r^{p-1}}h_{3}^{r^{p-2}}\neq h_1^{r^{p-1}}h_{4}^{r^{p-2}},\ \ \ \  h_2^{r^{p-1}}h_{3}^{r^{p-2}}\neq h_2^{r^{p-1}}h_{4}^{r^{p-2}}.$$
Since $|\Pi^2(\mathbf{B_{1,2}})|=3$, we have either $h_1^{r^{p-1}}h_{3}^{r^{p-2}}= h_2^{r^{p-1}}h_{4}^{r^{a(p-2)}}$ or $h_1^{r^{p-1}}h_{4}^{r^{p-2}}= h_2^{r^{p-1}}h_{3}^{r^{p-2}}$. Thus $(h_{4}h_{3}^{-1})^{r^{p-2}}=(h_1h_2^{-1})^{r^{p-1}}$ or $(h_{3}h_{4}^{-1})^{r^{p-2}}=(h_1h_2^{-1})^{r^{p-1}}$, whence
$$h_{3}h_{4}^{-1}\in \{(h_1h_2^{-1})^{r}, (h_1h_2^{-1})^{-r}\}.$$

Now, let $\mathbf{A'}=(A_2, A_1, A_3,\ldots, A_p)$. As above, we have
$$h_{1}h_{2}^{-1}\in \{(h_3h_4^{-1})^{r}, (h_3h_4^{-1})^{-r}\}.$$
Thus, $h_{1}h_{2}^{-1}\in \{(h_3h_4^{-1})^{r}, (h_3h_4^{-1})^{-r}\}\subseteq \{(h_1h_2^{-1})^{r^{2}}, (h_1h_2^{-1})^{-r^{2}}\}$. If $h_{1}h_{2}^{-1}=(h_1h_2^{-1})^{r^{2}}$, then $(h_1h_2^{-1})^{r^{2}-1}=1$; if $h_{1}h_{2}^{-1}=(h_1h_2^{-1})^{-r^{2}}$, then $(h_1h_2^{-1})^{r^{2}+1}=1$, and thus $(h_1h_2^{-1})^{r^{4}-1}=1$. Since $p\geq 3$, we have $p\nmid 4$. By Lemma~\ref{mequiv1}~(i), $\mbox{gcd}(r^{4}-1,m)=1$. So $h_1h_2^{-1}=1$ and thus $h_1=h_2$, yielding a contradiction. Hence, we must have $p=2$.

We now verify the second part of the conclusion. By \eqref{hT}, $\mathsf h(T)\leq p-1=1$, so $T$ is a subset of $G$ and $\supp(T)=\{g_1, g_2, g_3\}$, implying that $\{h_1,h_2,h_3\}$ is a subset of $N$. Note that
\begin{align*}
\Pi_2(T)=&\{g_1g_2, g_2g_1, g_1g_3, g_3g_1, g_2g_3, g_3g_2\}\\
=&\{h_1^{-1}h_2, h_2^{-1}h_1,  h_1^{-1}h_3, h_3^{-1}h_1, h_2^{-1}h_3, h_3^{-1}h_2\}\\
=&(\{h_1,h_2,h_3\}\{h_1^{-1},h_2^{-1},h_3^{-1}\})\setminus\{1\}.
\end{align*}
Let $H=\mbox{stab}(\{h_1,h_2,h_3\}\{h_1^{-1},h_2^{-1},h_3^{-1}\})$. Since $|H||m$ and $\mbox{gcd}(p,m)=\mbox{gcd}(2,m)=1$, we have $|H|$ is odd. If $H=\{1\}$, then by Lemma~\ref{Kneser},
$$|\{h_1,h_2,h_3\}\{h_1^{-1},h_2^{-1},h_3^{-1}\}|\geq |\{h_1,h_2,h_3\}|+|\{h_1^{-1},h_2^{-1},h_3^{-1}\}|-1=5.$$
Thus $|\Pi_2(T)|\geq |\{h_1,h_2,h_3\}\{h_1^{-1},h_2^{-1},h_3^{-1}\}|-1=4$, yielding a contradiction to the given condition that $|\Pi_2(T)|=2$. If $|H|> 3$, by Lemma~\ref{Kneser}, $|\{h_1,h_2,h_3\}\{h_1^{-1},h_2^{-1},h_3^{-1}\}|\geq |H|>3$. As above, $|\Pi_2(T)|>2$, yielding a contradiction. Therefore, $|H|=3$. If $\{h_1,h_2,h_3\}$ is not contained in a single coset of $H$, then $|\{h_1,h_2,h_3\}H|\geq 2|H|$ and $|\{h_1^{-1},h_2^{-1},h_3^{-1}\}H|\geq 2|H|$. By Lemma ~\ref{Kneser}, $|\{h_1,h_2,h_3\}\{h_1^{-1},h_2^{-1},h_3^{-1}\}|\geq 3|H|=9$, yielding a contradiction. Thus $\{h_1,h_2,h_3\}$ is contained in a single coset of $H$. Therefore, $\Pi_2(T)=(\{h_1,h_2,h_3\}\{h_1^{-1},h_2^{-1},h_3^{-1}\})\setminus\{1\}=H\setminus \{1\}$ as desired. This completes the proof of the lemma.
\end{proof}

We are now ready to prove our first main result.\\

\noindent {\bf Proof of Theorem \ref{inversed(G)}}
\\

$(a)\Rightarrow (b)$. Let $S$ be a product-one free sequence over $G$ of length $\ell=m+p-2$ and let $T=S\bm\cdot S_N^{[-1]}$. We first prove the following claim.
\\

\noindent {\bf Claim 1.} If $G\not\cong C_2\ltimes C_3$, then $\varphi(T)$ is product-one free over $G/N\cong C_p$  and thus  $|T|\leq p-1$.
\\

Assume to the contrary that $\varphi(T)$ is not product-one free. Then there exists a subsequence $T'\mid T$ such that $\varphi(T')$ is a product-one sequence over $G/N$ with maximal length. Let $W=T\bm\cdot T'^{[-1]}$, then $W$ is product-one free over $G/N$. Let $T'=T_1\bm\cdot \ldots \bm\cdot T_t$ where $\varphi(T_i)$ is a minimal product-one subsequence over $G/N$ for all $i\in[1,t]$. Then
$$T=T_1\bm\cdot \ldots \bm\cdot T_t \bm\cdot W.$$
Since $\mathsf d(G/N)=p-1$, we have $|T_i|=|\varphi(T_i)|\leq p$ for all $i\in[1,t]$ and $|W|=|\varphi(W)|\leq p-1$.

We first prove that
$$|S_N|=0\mbox{\ \ \ and \ \ \ }S=T=(x^ah_1)\bm\cdot\ldots \bm\cdot (x^ah_{\ell}),$$
where $a\in [1,p-1]$ and $h_i\in N$ for every $i\in[1,\ell]$. Moreover, we have
\begin{align}\label{PipT0}
|\Pi_p(T_0)|=p
\end{align}
for every subsequence $T_0\mid S$ with length $2p-1$.

Since $1\notin \Pi(S)$, we have $1\notin\pi(T_i)$, so if $u_i\in\pi(T_i)$, then $u_i\neq 1$ where $i\in[1,t]$. By Lemma~\ref{conjugationone},
$$\pi(T_i)\supseteq \{u_i^{r^{s_{i1}}},u_i^{r^{s_{i2}}},\ldots,u_i^{r^{s_{it_i}}}\}$$
for all $i\in [1,t]$, $t_i=|T_i|$, and moreover, $u_i^{s_{ij}}\neq u_i^{s_{ik}}$ for $1\leq j<k\leq t_i$. Let $$A_i=\{u_i^{r^{s_{i1}}},u_i^{r^{s_{i2}}},\ldots,u_i^{r^{s_{it_i}}}\}\mbox{ \ \ \ and \ \ \ } B=\Pi(S_N)$$
where $i\in [1,t]$. Then $|A_i|=|T_i|$. Let
$$\mathbf{A}=(\pi(T_1),\pi(T_2),\ldots,\pi(T_t), B),$$
Clearly, $\Pi(\mathbf{A})\subseteq \Pi(S)\cap N$. Since $1\notin \Pi(S)$, we have $1\notin B$ and $1\notin \Pi(\mathbf{A})$. By Corollary~\ref{zerofree}, $|B|\geq |S_N|$. Moreover, by Lemma~\ref{productonefree},
$$m-1\geq |\Pi(\mathbf{A})|\geq \sum_{i=1}^t|\pi(T_i)|+|B|\geq \sum_{i=1}^t|A_i|+|S_N|= \sum_{i=1}^t|T_i|+|S_N|=|S|-|W|\geq m-1.$$
Thus, we have $|S|-|W|=m-1$ and $|\pi(T_i)|=|A_i|$. Therefore,
$$|W|=p-1\mbox{ \ \ and \ \ }\pi(T_i)=A_i.$$
Since $\varphi (W)$ is a product-one free sequence of length $p-1$ over $G/N\cong C_p$, by Lemma~\ref{n-1} we have $|\supp(\varphi (W))|=1$. Note that $\varphi(T')$ is a product-one subsequence of $\varphi(T)$ over $G/N$ with maximal length $|\varphi(T')|=|\varphi(T\bm\cdot W^{[-1]})|=|\varphi(T)|-\mathsf d(G/N)$. By Lemma~\ref{normalsequence}, we conclude that if $\varphi(g)\in \supp(\varphi (T))$ for some $g\in G\setminus N$, then $\varphi(g)\in \supp(\varphi (W))$. Since $|\supp(\varphi (W))|=1$, we have
\begin{align}\label{suppvarphiT}
|\supp(\varphi (T))|=1.
\end{align}
Hence, $T$ is a subsequence over $N_a=x^aN$ for some $a\in [1,p-1]$. Since $\varphi (T_i)$ is a minimal product-one sequence over $G/N\cong C_p$, we have $|T_i|=|\varphi (T_i)|=p$ for every $i\in [1,t]$. If $|S_N|\geq 1$, then let $u\mid S_N$. By Lemma~\ref{conjugationtwo}, we have
$$\Pi(T_1\bm\cdot u)\cap N\supseteq \pi(T_1)\{u\}\supseteq \pi(T_1)\{u, u^{r},\ldots, u^{r^{p-1}}\}.$$
Since $1\notin \Pi(T_1\bm\cdot u)$, we have $1\notin\pi(T_1)\{u, u^{r},\ldots, u^{r^{p-1}}\}$. By Lemma~\ref{productonefree}, we have
$$|\Pi(T_1\bm\cdot u)\cap N|\geq |\pi(T_1)\{u, u^{r},\ldots, u^{r^{p-1}}\}|\geq |\pi(T_1)|+|\{u, u^{r},\ldots, u^{r^{p-1}}\}|\geq 2p.$$
Note that
\begin{align*}
\Pi(S)\cap N&\supseteq\Pi(T_1\bm\cdot \ldots \bm\cdot T_t\bm\cdot S_N)\cap N\\
&\supseteq(\Pi(T_1\bm\cdot u)\cap N)(\Pi(T_2\bm\cdot \ldots \bm\cdot T_t\bm\cdot S_N\bm\cdot u^{[-1]})\cap N)\\
&\supseteq(\Pi(T_1\bm\cdot u)\cap N)(\Pi(T_2)\cap N) \ldots (\Pi(T_t)\cap N)(\Pi(S_N\bm\cdot u^{[-1]})\cap N)\\
&\supseteq(\Pi(T_1\bm\cdot u)\cap N)\pi(T_2)\ldots \pi(T_t)\Pi(S_N\bm\cdot u^{[-1]}).
\end{align*}
By Lemma~\ref{productonefree} and Corollary~\ref{zerofree},
\begin{align*}
|\Pi(S)\cap N|&\geq |\Pi(T_1\bm\cdot u)\cap N|+|\pi(T_2)|+\ldots +|\pi(T_t)|+|\Pi(S_N\bm\cdot u^{[-1]})|\\
&\geq 2p+(\sum_{i=2}^{t}|T_i|+(|S_N|-1))\\
&\geq |T'|+\sum_{i=1}^{t}|T_i|+|S_N|\\
&=m+p-2\\
&\geq m,
\end{align*}
so $\Pi(S)\supseteq N\ni 1$, yielding a contradiction.

Hence $|S_N|=0$ and $S=T$. By \eqref{suppvarphiT}, we have $|\supp(\varphi (S))|=|\supp(\varphi (T))|=1$, whence
$$S=(x^ah_1)\bm\cdot\ldots \bm\cdot (x^ah_{\ell}),$$
where $a\in [1,p-1]$ and $h_i\in N$ for every $i\in[1,\ell]$. Let $T_0\mid S$ be any subsequence of $S$ with length $2p-1$. We obtain a factorization of $S$ such that
$$S=T_0\bm\cdot T_1\bm\cdot \ldots \bm\cdot T_{(m-1)/p-1}$$
with $|T_i|=p$ for all $i\in [1, (m-1)/p-1]$. Note that $\Pi_p(T_0)\subseteq N$ and
$$\Pi(S)\cap N\supseteq\Pi_p(T_0)\pi(T_1)\ldots \pi(T_{(m-1)/p-1}).$$
Since $1\notin \Pi(S)$, we have $1\notin \Pi_p(T_0)$ and $1\notin \pi(T_i)$ for all $i\in [1,(m-1)/p-1]$. Let $T_0'$ be a subsequence of $T_0$ with length $p$. Then $\Pi_p(T_0)\supseteq \pi(T_0')$. By Lemma~\ref{conjugationone}, we have $|\Pi_p(T_0)|\geq |\pi(T_0')|\geq p$ and $|\pi(T_i)|\geq p$ for all $i\in [1,(m-1)/p-1]$. Thus by Lemma~\ref{productonefree}, we have
\begin{align*}
m-1&\geq |\Pi(S)\cap N|\\
&\geq |\Pi_p(T_0)\pi(T_1)\ldots \pi(T_{(m-1)/p-1})|\\
&\geq |\Pi_p(T_0)|+|\pi(T_1)|+\ldots +|\pi(T_{(m-1)/p-1})|\\
&\geq p+p((m-1)/p-1)\\
&\geq m-1.
\end{align*}
Therefore, we have $|\Pi_p(T_0)|=p$.

By Lemma~\ref{2p-1}, we have $p=2$ and $N$ has a subgroup $H$ of order $3$, so $3||N|$. Since $G\not\cong C_2\ltimes C_3$, we have $m\geq 9$. Thus, there exist two disjoint subsequences $T_0$ and $T_0'$ of $S$ with $|T_0|=|T_0'|=3$. By \eqref{PipT0}, $|\Pi_2(T_0)|=|\Pi_2(T_0')|=2$, and thus by Lemma~\ref{2p-1}, we have $\Pi_2(T_0)=\Pi_2(T_0')=H\setminus \{1\}$. Clearly, $1\in \Pi_2(T_0)\Pi_2(T_0')\subseteq \Pi(S)$, yielding a contradiction. This completes the proof of our claim.

(i) Since $G\not\cong C_2\ltimes C_3$, by Claim 1, we have $|T|\leq p-1$ and thus $|S_N|\geq m-1$. Since $1\notin\Pi(S_N)$ and $N\cong C_m$, we have $|S_N|\leq m-1$. Hence $|S_N|=m-1$ and thus $|T|=|S|-|S_N|=p-1$. By Claim 1, $\varphi(T)$ is product-one free over $G/N$. Therefore, by Lemma~\ref{n-1}, $|\supp(\varphi(T))|=1$ and thus $T$ is contained in a single coset of $N$, whence $T=(x^ay^{b_1})\bm\cdot\ldots\bm\cdot (x^ay^{b_{p-1}})$ where $a\in [1,p-1]$ and $b_i\in [0,p-1]$ for every $i\in[1,p-1]$. Since $1\notin\Pi(S_N)$ and $|S_N|=m-1$, by Lemma~\ref{n-1}, $S_N=(y^c)^{[m-1]}$ for some $c\in [0,m-1]$ with $\mbox{gcd}(c,m)=1$. Putting this together, we obtain
$$S=T\bm\cdot S_N=(x^ay^{b_1})\bm\cdot\ldots\bm\cdot (x^ay^{b_{p-1}})\bm\cdot (y^c)^{[m-1]}$$
as desired.

(ii) If $\varphi(T)$ is not product-one free over $G/N$, then as in the proof of the claim, we have $S=(xh_1)\bm\cdot (xh_{2}) \bm\cdot (xh_{3})$. Since $1\notin \Pi(S)$ and $\mbox{ord}(xh_{i})=2$ for all $i\in [1,3]$, we have $\mathsf h(S)=1$ and thus
$$S=x\bm\cdot xy\bm\cdot xy^2.$$
If $\varphi(T)$ is product-one free over $G/N$, then $|T'|\leq 1$. As in (i), we have
$$S= xy^{b}\bm\cdot (y^{c})^{[2]}$$
for $b\in [0,2]$ and $c\in[1,2]$. \\

$(b)\Rightarrow (a)$. (i) Let $S'$ be any subsequence of $S$. To show that $S$ is product-one free, i.e., $1\notin \Pi(S)$, it suffices to show that $1\notin \pi(S')$. Write $S_1'\mid (x^ay^{b_1})\bm\cdot\ldots\bm\cdot (x^ay^{b_{p-1}})$ and $S_2'\mid (y^c)^{[m-1]}$. If $S_1'\neq \emptyset$, then since $a\in[1,p-1]$ and $1\neq |S_1'|\leq p-1$, we have that $\varphi(S')$ is not a product-one sequence over $G/N$, i.e. $\varphi(\pi(S'))=\pi(\varphi(S'))\neq 1$. Thus $\pi(S')\cap N=\emptyset$. Therefore, $1\notin\pi(S')$. If $S_1'=\emptyset$, then $S'\mid (y^c)^{[m-1]}$. Clearly, $1\notin\pi(S')$ as $\mbox{gcd}(c,m)=1$. Thus $S$ is product-one free.

(ii) If $G\cong C_2\ltimes C_3$ and $S=x\bm\cdot xy\bm\cdot xy^2$ or $S= xy^{b}\bm\cdot (y^{c})^{[2]}$ for $b\in [0,2]$ and $c\in[1,2]$, then it is easy to check that $1\notin\Pi(S)$, so $S$ is product-one free.
\qed

\section{$|G|$-Product-One Free Sequences}
In this section, we characterize the structure of $|G|$-product-one free sequences of length $\mathsf E(G)-1$ and provide a proof for Theorem~\ref{inverseE(G)}. We first prove a few lemmas and we fix some notation which will be used throughout this section. Let $\mathbf{A}=(A_1, A_2, \ldots, A_{\ell})$ be a sequence of subsets of $N$ with $A_i=\{u_i^{r^{k_1}},u_i^{r^{k_2}},\ldots,u_i^{r^{k_t}}\}$ where $0=k_1<k_2\ldots<k_t\leq p-1$, $t\in [1, p]$, $u_i\in N$ and $i\in [1,\ell]$. Let $M=\mbox{stab}(\Pi^{v-1}(\mathbf{A}))$ where $v\in[2,\ell+1]$. Then $M$ is a subgroup of $N$. Let
\begin{center}
$I_M$ be the subset of $[1,\ell]$ such that $i\in I_M$ if and only if $A_i\subseteq M$ for $i\in [1,\ell]$.
\end{center}
Let $Q\in N/M$ be a coset of $M$,
$$V_Q=\{i\in[1,\ell]: A_i\cap Q\neq \emptyset\}$$
and $$\mu=|\{Q\in N/M: |V_Q|\geq v\}|.$$

\begin{lemma}\label{IM}
\begin{itemize}
\item[(i)] $V_M=I_M$. Moreover, if $Q\neq M$, then $V_Q\cap V_M=\emptyset$.
\item[(ii)] If $Q\neq M$ and $i\in V_Q$ for some $i$, then all $t$ elements of $A_i$ are in $t$ different cosets of $M$.
\item[(iii)] If $t=p$, $\mu =1$, and $R$ is the unique coset of $M$ such that $|V_R|\geq v$, then $R=M$.
\end{itemize}
\end{lemma}
\begin{proof}
(i) By definitions of $I_M$ and $V_M$, we have $V_M\supseteq I_M$. It follows from Lemma~\ref{basic}~(i) that $V_M\subseteq I_M$, so $V_M=I_M$. Moreover, if $Q\neq M$ and $V_Q\cap V_M\neq \emptyset$, then $V_Q\cap I_M\neq \emptyset$. Thus, there exists an integer $i\in [1,\ell]$ such that $A_i\cap Q\neq \emptyset$ and $A_i\subseteq M$, yielding a contradiction. Therefore, if $Q\neq M$ then $V_Q\cap V_M=\emptyset$.

(ii) Given that $Q\neq M$, by (i) $V_Q\cap V_M=\emptyset$. Since $i\in V_Q$, we have $i\notin V_M$ and thus $1\notin A_i$. By Lemma~\ref{basic}~(iii), $|A_i|=t$. Let $\alpha\neq \beta$ be any two elements of $A_i$. If $\alpha$ and $\beta$ are in the same coset of $M$, then $\alpha\beta^{-1}\in M$. By Lemma~\ref{basic}~(ii), $\alpha\in M$. Thus $i\in V_M$, yielding a contradiction. Therefore, all $t$ elements of $A_i$ are in $t$ different cosets of $M$.

(iii) Assume to the contrary that $R\neq M$. Then $R=\alpha M$ for some $\alpha\in N\setminus M$. Let $\alpha_i\in A_i \cap R=A_i \cap \alpha M$ for all $i\in V_R$. Since $\alpha_i\in A_i=\{u_i, u_i^{r}\ldots, u_i^{r^{p-1}}\}$, $\alpha_i^{r}\in A_i$. Thus $\alpha_i^{r}\in A_i\cap \alpha^{r}M$ for all $i\in V_R$, implying that for all $i\in V_R$, $i\in V_{\alpha^{r}M}$ and thus $|V_{\alpha_i^rM}|\geq |V_R|$. Since $\alpha\notin M$, by Lemma~\ref{basic}~(ii), $\alpha M\neq \alpha^r M$. So we have found another coset $\alpha^r M(\neq R)$ such that $|V_{\alpha_i^rM}|\geq |V_R|\geq v$, yielding a contradiction to $\mu=1$. Hence we must have $R=M$.
\end{proof}

\begin{lemma}\label{mu}
\begin{itemize}
\item[(i)] If $\mu=0$, then $|\Pi^{v-1}(\mathbf{A})|\geq |M|(2-v+\ell t-(t-1)|I_M|)$.
\item[(ii)] If $\mu\geq 2$, then $|\Pi^{v-1}(\mathbf{A})|\geq v|M|$.
\item[(iii)] If $\mu=1$ and $R=M$, then $|\Pi^{v-1}(\mathbf{A})|\geq |M|(1+t(\ell-|I_M|))$.
\item[(iv)] If $\mu=1$ and $R\neq M$, then $|\Pi^{v-1}(\mathbf{A})|\geq |M|(1+(t-1)\ell-(t-2)|I_M|)$.
\end{itemize}
\end{lemma}
\begin{proof}
(i) If $\mu=0$, then by Lemma~\ref{genKneser},
\begin{align*}
|\Pi^{v-1}(\mathbf{A})|&\geq |M|\Big(2-v+\sum_{Q\in N/M}\min\big\{v-1,|\{i\in[1,\ell]: A_i\cap Q\neq \emptyset\}|\big\}\Big)\\
&=|M|\Big(2-v+\sum_{Q\in N/M}|\{i\in[1,\ell]: A_i\cap Q\neq \emptyset\}|\Big)\ \ \   (\mbox{as } \mu=0)\\
&=|M|\Big(2-v+|V_M|+\sum_{Q\in N/M, Q\neq M}|\{i\in[1,\ell]: A_i\cap Q\neq \emptyset\}|\Big)\\
&=|M|\Big(2-v+|I_M|+\sum_{Q\in N/M, Q\neq M}\sum_{i\in[1,\ell], A_i\cap Q\neq \emptyset}1\Big)\ \ (\mbox{by Lemma~\ref{IM}~(i)})\\
&=|M|\Big(2-v+|I_M|+\sum_{i\in[1,\ell]\setminus I_M}\sum_{Q\in N/M, A_i\cap Q\neq \emptyset}1\Big)\\
&=|M|(2-v+t(\ell-|I_M|)+|I_M|)\ \ (\mbox{by Lemma~\ref{IM}~(ii)})\\
&=|M|(2-v+\ell t-(t-1)|I_M|).
\end{align*}
(ii) If $\mu\geq 2$, then by Lemma~\ref{genKneser},
\begin{align*}
|\Pi^{v-1}(\mathbf{A})|\geq &|M|\Big(2-v+\sum_{Q\in N/M}\min\big\{v-1,|\{i\in[1,\ell]: A_i\cap Q\neq \emptyset\}|\big\}\Big)\\
=&|M|(2-v+2(v-1))=v|M|.
\end{align*}
(iii) If $\mu=1$ and $R=M$, then by Lemma~\ref{genKneser},
\begin{align*}
|\Pi^{v-1}(\mathbf{A})|\geq &|M|\Big(2-v+\sum_{Q\in N/M}\min\big\{v-1,|\{i\in[1,\ell]: A_i\cap Q\neq \emptyset\}|\big\}\Big)\\
=&|M|\Big(2-v+v-1+\sum_{Q\in N/M, Q\neq M}|\{i\in[1,\ell]: A_i\cap Q\neq \emptyset\}|\Big)\\
=&|M|\Big(1+\sum_{Q\in N/M, Q\neq M}\sum_{i\in[1,\ell], A_i\cap Q\neq \emptyset}1\Big)\\
=&|M|\Big(1+\sum_{i\in[1,\ell]\setminus I_M}\sum_{Q\in N/M, A_i\cap Q\neq \emptyset}1\Big)\\
=& |M|(1+t(\ell-|I_M|))\ \ \ \ \ \  \ \ (\mbox{by Lemma~\ref{IM}~(ii)}).
\end{align*}
(iv) If $\mu=1$ and $R\neq M$, then by Lemma~\ref{IM}~(i), $V_R\cap I_M=\emptyset$. Therefore, $|I_M|+|V_R|\leq \ell$. By Lemma~\ref{genKneser},
\begin{align*}
|\Pi^{v-1}(\mathbf{A})|\geq &|M|\Big(2-v+\sum_{Q\in N/M}\min\big\{v-1,|\{i\in[1,\ell]: A_i\cap Q\neq \emptyset\}|\big\}\Big)\\
=&|M|\Big(2-v+v-1+\sum_{Q\in N/M, Q\neq R}|\{i\in[1,\ell]: A_i\cap Q\neq \emptyset\}|\Big) \\
=&|M|\Big(1-|V_R|+|V_M|+\sum_{Q\in N/M, Q\neq M}|\{i\in[1,\ell]: A_i\cap Q\neq \emptyset\}|\Big)\\
=&|M|\Big(1-|V_R|+|I_M|+\sum_{Q\in N/M, Q\neq M}\sum_{i\in[1,\ell], A_i\cap Q\neq \emptyset}1\Big)\\
=&|M|\Big(1-\ell+2|I_M|+\sum_{i\in[1,\ell]\setminus I_M}\sum_{Q\in N/M, A_i\cap Q\neq \emptyset}1\Big)\\
=& |M|(1-\ell+2|I_M|+t(\ell-|I_M|))\ \ \ \ \ \  \ \ (\mbox{by Lemma~\ref{IM}~(ii)})\\
=& |M|(1+(t-1)\ell-(t-2)|I_M|).
\end{align*}
This completes the proof of the lemma.
\end{proof}

\begin{lemma}\label{1inpiU}
Let $S$ be a sequence over $N$ with length $m+p-2$ and $T$ be a sequence with length $p$ over $N_i$ for some $i\in [1,p-1]$. There exists a subsequence $U\mid S$ with $p||U|$ such that either $1\in \pi(U)$ or $1\in \pi(T\bm\cdot U)$.
\end{lemma}

\begin{proof}
Let $S=u_1\bm\cdot \ldots \bm\cdot u_{\ell}$, where $\ell=m+p-2$ and let $n=\mathsf v_1(S)$, i.e., the number of times $1$ occurs in $S$. Let $A_j=\{u_j, u_j^{r}, \ldots,  u_j^{r^{p-1}}\}$ for $j\in[1,\ell-n]$ and $A_j=\{1\}$ for $j\in[\ell-n+1,\ell]$. Clearly, $\ell\geq m-1$. By Lemma~\ref{conjugationtwo},
\begin{align}\label{pU}
\pi(T\bm\cdot u_{i_1}\bm\cdot\ldots\bm\cdot u_{i_{m-1}})\supseteq \pi(T)A_{i_1}\ldots A_{i_{m-1}}
\end{align}
for every $(m-1)$-subset $\{i_1,\ldots,i_{m-1}\}\subseteq [1,\ell]$. Let $\mathbf{A}=(A_1,\ldots, A_{\ell})$ and $M=\mbox{stab}(\Pi^{m-1}(\mathbf{A}))$. Recall that $I_M$ is the subset of $[1,\ell]$ such that $j\in I_M$ if and only if $A_j\subseteq M$. If $|I_M|\geq p|M|+|M|-1$, then by Lemma~\ref{knzerosumfree}~(i), there exists a subsequence of $U\mid S_M$ with length $|U|=p|M|$ such that $1\in\pi(U)$ as desired.

Next, we assume that $|I_M|\leq p|M|+|M|-2$ and will prove $|\Pi^{m-1}(\mathbf{A})|\geq m$ (i.e., $\Pi^{m-1}(\mathbf{A})=N$). Clearly, if $M=N$, then $\Pi^{m-1}(\mathbf{A})=N$ and thus $|\Pi^{m-1}(\mathbf{A})|\geq m$ as desired. We may always assume that $M\lneq N$. Let
$$\mu=|\{Q\in N/M: |V_Q|\geq m\}|.$$
If $\mu=0$, then by Lemma~\ref{mu}~(i),
\begin{align*}
|\Pi^{m-1}(\mathbf{A})|&\geq  |M|(2-m+p\ell-(p-1)|I_M|)\\
&\geq |M|(p\ell-m-(p-1)(p|M|+|M|-2)+2)\\
&\ \ \ \ \ \ \ \ \ \ \ \ (\mbox{as } |I_M|\leq p|M|+|M|-2)\\
&= |M|(p-1)(m-(p+1)|M|)+p^2|M|     \ \ \         (\mbox{as } \ell=m+p-2)\\
&= (|M|(p-1)-1)(m-(p+1)|M|)+m+(p^2-p-1)|M|     \\
&\geq m+(p^2-p-1)|M|     \ \ \ \ \         (\mbox{as } m/|M|\geq p+1)\\
&\geq m.
\end{align*}
If $\mu\geq 2$, then by Lemma~\ref{mu}~(ii),
\begin{align*}
|\Pi^{m-1}(\mathbf{A})|\geq m|M|\geq m.
\end{align*}
If $\mu=1$, then let $R\in N/M$ be the unique coset of $M$ such that $|V_R|\geq m$. By Lemma~\ref{IM}~(iii), $R=M$, so
$$|V_M|=|V_R|\geq m.\ \ \ \ \ \ \ \ \ \ (*)$$
However, by Lemma~\ref{IM}~(i), we have $|V_M|=|I_M|\leq p|M|+|M|-2<m$, yielding a contradiction to $(*)$. Thus $\mu\neq 1$.

In all possible cases, we have shown $|\Pi^{m-1}(\mathbf{A})|\geq m$ (i.e., $\Pi^{m-1}(\mathbf{A})=N$). Thus $1\in\pi(T)\Pi^{m-1}(\mathbf{A})=N$, whence there exists a $(m-1)$-subset $\{i_1,\ldots,i_{m-1}\}\subseteq [1,\ell]$ such that $1\in \pi(T)A_{i_1}\ldots A_{i_{m-1}}$. By \eqref{pU}, $1\in\pi(T\bm\cdot u_{i_1}\bm\cdot\ldots\bm\cdot u_{i_{m-1}})$. By Lemma~\ref{mequiv1}~(i), $p|m-1$. Let $U=u_{i_1}\bm\cdot\ldots\bm\cdot u_{i_{m-1}}$. Then $U\mid S$ with $p||U|$ such that $1\in \pi(T\bm\cdot U)$. This completes the proof of the lemma.
\end{proof}

\begin{lemma}\label{SNigeqp}
Let $S$ be a sequence over $G$ with length $mp+m+p-3$ such that $1\notin \Pi_{mp}(S)$. If $|S_{N_i}|\geq p$ for some $i\in [1,p-1]$ where $N_i=x^iN$, then
$$G\cong C_2\ltimes C_3 \mbox{ and } S=x\bm\cdot xy \bm\cdot xy^{2}\bm\cdot  1^{[5]}.$$
\end{lemma}
\begin{proof}
Let $T_0$ be any subsequence of $S_{N_i}$ with $|T_0|=p$. Then $\varphi(T_0)=(\bar{x}^i)^{[p]}$ is a minimal product-one subsequence over $G/N$, where $\bar{x}=\varphi(x)$. Let $\ell$ be the maximal integer such that
\begin{align*}
S=T_0\bm\cdot T_1\bm\cdot\ldots\bm\cdot T_{\ell}\bm\cdot S',
\end{align*}
where $\varphi(T_j)$ is a product-one subsequence over $G/N$ with $|T_j|=p$ for every $j\in[1,\ell]$. By rearranging the order of $T_1, T_2, \ldots, T_{\ell}$ if necessary, we may assume that
$\pi(T_j)\neq \{1\}$ for all $j\in[1,v]$ and $\pi(T_j)=\{1\}$ for all $j\in [v+1,\ell]$. Since $G/N\cong C_p$, by the maximality of $\ell$, we conclude that $\varphi(S')$ has no product-one subsequence of length $p$. Thus
$$|S\bm\cdot (T_0\bm\cdot T_1\bm\cdot\ldots\bm\cdot T_{\ell})^{[-1]}|=|S'|=|\varphi(S')|\leq \mathsf E(G/N)-1=2p-2.$$
Therefore, $p\ell=|S|-|S'|-|T_0|\geq |S|-2p+2-p=mp+m-2p-1$. By Lemma~\ref{mequiv1}~(i), $m\equiv1 \pmod p$. Thus $\ell\geq m+(m-1)/p-1$ or $\ell=m+(m-1)/p-2$.
\\

\noindent {\bf Case 1.} $\ell\geq m+(m-1)/p-1$. We will find a contradiction and thus this case is impossible.

As $p\ell\leq |S|-|T_0|=mp+m-3$, we have $\ell\leq m+\frac{m-1}{p}-\frac{2}{p}<m+\frac{m-1}{p}$. Thus $\ell=m+(m-1)/p-1$. Let $1\neq u_j\in\pi(T_j)$ for every $j\in [1,v]$ and $1=u_j\in\pi(T_j)$ for every $j\in [v+1,\ell]$. Let $A_0=\pi(T_0)$, $A_j=\{u_{j},u_{j}^{r},\ldots,u_{j}^{r^{p-1}}\}$ for all $j\in [1,v]$, and $A_j=\{1\}$ for all $j\in [v+1,\ell]$. By Lemma~\ref{conjugationtwo},
$$\pi(T_0\bm\cdot T_{j_1}\bm\cdot\ldots\bm\cdot T_{j_{m-1}})\supseteq A_0A_{j_1}\ldots A_{j_{m-1}}$$
for every $(m-1)$-subset $\{j_1,\ldots,j_{m-1}\}\subseteq [1,\ell]$. Let $\mathbf{A}=(A_1,\ldots, A_{\ell})$. Then
$$\Pi_{mp}(S)\supseteq A_0\Pi^{m-1}(\mathbf{A}).$$
Let $M=\mbox{stab}(\Pi^{m-1}(\mathbf{A}))$. Recall that $I_M$ is the subset of $[1,\ell]$ such that $j\in I_M$ if and only if $A_j\subseteq M$. Since $M$ is a cyclic subgroup of $N$, if $|I_M|\geq (m/|M|) |M|+|M|-1=m+|M|-1$, then by Lemma~\ref{knzerosumfree}~(i), there exists a subset $\{j_1,\ldots, j_m\}\subseteq I_M$ such that $1\in \pi(T_{j_1}\bm\cdot \ldots\bm\cdot T_{j_m})$. Since $|T_{j_1}\bm\cdot \ldots\bm\cdot T_{j_m}|=mp$, $1\in\Pi_{mp}(S)$, yielding a contradiction. Thus $|I_M|\leq m+|M|-2$.

We now show $|\Pi^{m-1}(\mathbf{A})|\geq m$. As in the proof of Lemma~\ref{1inpiU}, we may always assume that $M\lneq N$. Let $$\mu=|\{Q\in N/M: |V_Q|\geq m\}|.$$
If $\mu=0$, then by Lemma~\ref{mu}~(i), we have
\begin{align*}
|\Pi^{m-1}(\mathbf{A})|\geq& |M|(2-m+p\ell-(p-1)|I_M|)\\
\geq& 2-m+(mp+m-p-1)-(p-1)(m-1)\\
&(\mbox{as } p\ell=mp+m-p-1 \mbox{ and } |I_M|=|V_M|\leq m-1)\\
=& m.
\end{align*}
If $\mu\geq 2$, then by Lemma~\ref{mu}~(ii),
\begin{align*}
|\Pi^{m-1}(\mathbf{A})|\geq m|M|\geq m.
\end{align*}
If $\mu=1$, then let $R\in N/M$ be the unique coset of $M$ such that $|V_R|\geq m$. By Lemma~\ref{IM}~(iii), $R=M$. Note that $M\neq \{1\}$, as $|I_{\{1\}}|\leq m-1$. Thus by Lemma~\ref{mu}~(iii),
\begin{align*}
|\Pi^{m-1}(\mathbf{A})|\geq& |M|(1+p(\ell-|I_M|))\\
\geq& |M|(1+mp+m-p-1-p(m+|M|-2))\\
&(\mbox{as } p\ell=mp+m-p-1\mbox{ and } |I_M|\leq m+|M|-2)\\
=& |M|(m-p(|M|-1))\\
=& (m-p|M|)(|M|-1)+m \\
\geq& m\ \ \ \ \ (\mbox{as } m\geq (p+1)|M|).
\end{align*}
In all the cases, we have shown $|\Pi^{m-1}(\mathbf{A})|\geq m$. Thus $1\in N\subseteq \Pi^{m-1}(\mathbf{A})\subseteq A_0\Pi^{m-1}(\mathbf{A})\subseteq\Pi_{mp}(S)$, yielding a contradiction, and we are done.
\\

\noindent {\bf Case 2.} $\ell=m+(m-1)/p-2$.

Then $|S\bm\cdot T_0^{[-1]}|=mp+m-3=\ell p+2p-2$. By the maximality of $\ell$, $\varphi(S\bm\cdot T_0^{[-1]})$ has no product-one subsequence with length $(\ell+1) p$ over $G/N\cong C_p$. By Lemma~\ref{knzerosumfree}~(ii), $$\varphi(S\bm\cdot T_{0}^{[-1]})=(\bar{x}^{a_1})^{[k_1p-1]}\bm\cdot (\bar{x}^{a_2})^{[k_2p-1]},$$
where $\bar{x}=\varphi(x)$, $k_1+k_2=m+(m-1)/p$ and $0\leq a_1<a_2\leq p-1$. Let
$$S\bm\cdot T_0^{[-1]}=S_1\bm\cdot S_2$$
such that
$$\varphi(S_1)=(\bar{x}^{a_1})^{[k_1p-1]}\mbox{\ \ \  and \ \ \ }\varphi(S_2)=(\bar{x}^{a_2})^{[k_2p-1]}.$$
Then we may factor $S_1$ as follows
$$S_1=S_{11}\bm\cdot\ldots\bm\cdot S_{1(v_1+w_1)}\bm\cdot W_1,$$
where $\pi(S_{1i})\neq \{1\}$ for $i\in[1,v_1]$, $\pi(S_{1i})=\{1\}$ for $i\in[v_1+1,v_1+w_1]$, $|S_{1i}|=p$ for $i\in [1,v_1+w_1]$ and $|W_{1}|=p-1$. Similarly, let
$$S_2=S_{21}\bm\cdot\ldots\bm\cdot S_{2(v_2+w_2)}\bm\cdot W_2,$$
where $\pi(S_{2i})\neq \{1\}$ for $i\in[1,v_2]$, $\pi(S_{2i})=\{1\}$ for $i\in[v_2+1,v_2+w_2]$, $|S_{2i}|=p$ for $i\in [1,v_2+w_2]$ and $|W_{2}|=p-1$. Then, we get a new factorization of $S$ as follows
\begin{align}\label{factorization}
S=T_0&\bm\cdot S_{11}\bm\cdot\ldots\bm\cdot S_{1v_1}\bm\cdot S_{21}\bm\cdot\ldots\bm\cdot S_{2v_2}\\
& \bm\cdot S_{1(v_1+1)}\bm\cdot\ldots\bm\cdot S_{1(v_1+w_1)}\bm\cdot S_{2(v_2+1)}\bm\cdot\ldots\bm\cdot S_{2(v_2+w_2)}\bm\cdot (W_1\bm\cdot W_2).  \notag
\end{align}
By using this new factorization of $S$, as in Case $1$ we can choose $A_1,\ldots, A_{\ell}$ from \\ $\pi(S_{11}),\ldots, \pi(S_{2(v_2+w_2)})$ with the same properties such that $$w_1+w_2=|I_{\{1\}}|,$$
where $\ell=v_1+v_2+w_1+w_2=m+(m-1)/p-2$.

Let $\mathbf{A}=(A_1,\ldots,A_{\ell})$ and $M=\mbox{stab}(\Pi^{m-1}(\mathbf{A}))$. A similar computation with $p\ell=mp+m-2p-1$ shows that $|\Pi^{m-1}(\mathbf{A})|\geq m$, except for the case when $\mu=0$ and $M=\{1\}$, and this yields a contradiction. Next, we deal with the case when $\mu=0$ and $M=\{1\}$. Note that we always have $|I_{\{1\}}|\leq m-1$. According to $|I_{\{1\}}|$, we divide the remaining proof of Case 2 into the following three subcases. We will show that the first two subcases are impossible.
\\

\noindent {\bf Subcase 2.1.} $|I_{\{1\}}|\leq m-3$.

Since $M=\{1\}$, by Lemma~\ref{mu}~(i), we have
\begin{align*}
|\Pi^{m-1}(\mathbf{A})|\geq& 2-m+p\ell-(p-1)|I_{\{1\}}|\\
\geq& 2-m+(mp+m-2p-1)-(p-1)(m-3))\\
&(\mbox{as } p\ell=mp+m-2p-1 \mbox{ and } |I_{\{1\}}|\leq m-3)\\
=& m+p-2\geq m.
\end{align*}
As in Case 1, this yields a contradiction and so this subcase is impossible.
\\

\noindent {\bf Subcase 2.2.} $|I_{\{1\}}|=m-2$.

As in Subcase 2.1, we have
$$|\Pi^{m-1}(\mathbf{A})|\geq 2-m+p\ell-(p-1)|I_{\{1\}}|= m-1.$$
If $\pi(T_0)\neq \{1\}$, then by Lemma~\ref{conjugationone}, $|\pi(T_0)|\geq p$. Therefore, $|\pi(T_0)|+|\Pi^{m-1}(\mathbf{A})|>m$. By Lemma~\ref{A+B>G}, $\pi(T_0)\Pi^{m-1}(\mathbf{A})=N$ and thus $1\in\pi(T_0)\Pi^{m-1}(\mathbf{A})$, yielding a contradiction. Therefore,
$$\pi(T_0)= \{1\}.$$

We remark that for every factorization of $S$ of the form \eqref{factorization} with $T_0$ fixed, we always have
\begin{align}\label{m-2}
w_1+w_2=|I_{\{1\}}|=m-2.
\end{align}
For otherwise, if $|I_{\{1\}}|\leq m-3$, as in Subcase 2.1, we obtain a contradiction. If $|I_{\{1\}}|= m-1$, then $1\in \pi(T_0\bm\cdot S_{1(v_1+1)}\bm\cdot\ldots\bm\cdot S_{1(v_1+w_1)}\bm\cdot S_{2(v_2+1)}\bm\cdot\ldots\bm\cdot S_{2(v_2+w_2)})$. Since $|T_0\bm\cdot S_{1(v_1+1)}\bm\cdot\ldots\bm\cdot S_{1(v_1+w_1)}\bm\cdot S_{2(v_2+1)}\bm\cdot\ldots\bm\cdot S_{2(v_2+w_2)}|=mp$, we have $1\in\Pi_{mp}(S)$, yielding a contradiction as well.

Next, we prove that
\begin{align}
\mbox{ if } w_i\geq 1 \mbox{ for some } i\in [1,2], \mbox{ then } |\supp(S_i)|=1 \mbox{ and thus } v_i=0. \label{w_i>0}
\end{align}
We first show that $|\supp(W_i)|=1$, i.e., $W_i=g_i^{[p-1]}$ for some $g_i\in G$. Assume to the contrary that there exists $g_i'\bm\cdot g_i\mid W_i$ such that $g_i'\neq g_i$. Since $w_1\geq 1$, let $g\mid S_{i(v_i+1)}$. Since $g_i'\neq g_i$, without loss of generality, suppose $g\neq g_i$. Let $S_{i(v_i+1)}'=S_{i(v_i+1)}\bm\cdot g^{[-1]}\bm\cdot g_i$ and $W_{i}'=W_{i}\bm\cdot g_i^{[-1]}\bm\cdot g$. If $g_i\in N$, then clearly $\pi(S_{i(v_i+1)}')\neq \{1\}$. If $g_i\in G\setminus N$, then by Lemma~\ref{pit=1}, $\pi(S_{i(v_i+1)}')\neq \{1\}$. Thus there exists a new factorization of the form \eqref{factorization} with fixed $T_0$ such that $|I_{\{1\}}|=m-3$, yielding a contradiction to \eqref{m-2}. Therefore, $W_i=g_i^{[p-1]}$.

Next, we show $S_{i(v_i+1)}\bm\cdot \ldots\bm\cdot S_{i(v_i+w_i)}\bm\cdot W_i=g_i^{[w_ip+p-1]}$. Assume to the contrary that there exists $g_i'\mid S_{ij}$ such that $g_i'\neq g_i$ for some $j\in [v_i+1,v_i+w_i]$. Let $S_{ij}'=S_{ij}\bm\cdot g_i'^{[-1]}\bm\cdot g_i$ and $W_{i}'=W_{i}\bm\cdot g_i^{[-1]}\bm\cdot g_i'$. As above, $\pi(S_{ij}')\neq \{1\}$, yielding a contradiction. Therefore, $S_{i(v_i+1)}\bm\cdot \ldots\bm\cdot S_{i(v_i+w_i)}\bm\cdot W_i=g_i^{[pw_i+p-1]}$.

We now show that $S_i=g_i^{[(v_i+w_i)p+p-1]}$. Assume to the contrary that there exists $g_i'\mid S_{ij}$ for some $j\in [1,v_i]$ such that $g_i'\neq g_i$. Let $S_{ij}'=S_{ij}\bm\cdot g_i'^{[-1]}\bm\cdot g_i$ and $W_{i}'=W_{i}\bm\cdot g_i^{[-1]}\bm\cdot g_i'$. Then there exists a new factorization of the form \eqref{factorization} with fixed $T_0$ such that $w_i\geq 1$ and $|\supp(W_i')|=2$, hence yielding a contradiction to $|\supp(W_i)|=1$. Thus, $S_i=g_i^{[|S_i|]}$ and clearly $v_i=0$. This completes the proof of \eqref{w_i>0}.

Next, we prove that
$$a_1=0\mbox{ and }w_1=0.$$
If $a_1\neq 0$, then $S_N$ is empty and so $p\geq 3$. If both
$w_1\geq 1$ and $w_2\geq 1$, then by \eqref{w_i>0}, we have $v_i=0$ and $|S_i|=w_ip+p-1$ for $i\in [1,2]$. Therefore, $mp+m+p-3=|S|=|T_0|+|S_1|+|S_2|=p+(w_1p+p-1)+(w_2p+p-1)=3p-2+(w_1+w_2)p=mp+p-2$, yielding a contradiction. Therefore, without loss of generality, we may assume $w_2=0$, and then $w_1=m-2$. As above, $|S_1|=w_1p+p-1=mp-p-1$. Therefore, $|S_2|=|S|-|T_0|-|S_1|=mp+m+p-3-p-(mp-p-1)=m+p-2$. Note that $S_2$ is contained in exactly one coset of $N$ and $p\geq 3$. By Theorem~\ref{inversed(G)}, $S_2$ is not product-one free. Thus, there exists a product-one subsequence $T'\mid S_2$. Consequencly, $\varphi(T')$ is product-one over $G/N\cong C_p$. Since $\supp(\varphi(T'))=\{\bar{x}^{a_2}\}$ and $a_2\in [1,p-1]$, we have $p$ divides $|T'|$. Since $\pi(T_0)=\{1\}$ and $w_1=m-2$, we have that $T'\bm\cdot T_0\bm\cdot g_1^{[mp-|T_0|-|T'|]}$ is a product-one subsequence with length $mp$, yielding a contradiction and thus $a_1=0$.

Next, we show that $w_1=0$. Since $a_1=0$, we have $S_1=S_N$. If $w_1\geq 1$, then by \eqref{w_i>0} we have $S_1=g_1^{[|S_1|]}$ and $v_1=0$. Moreover, $g_1^{p}=1$ and so $g_1=1$. If $w_2\geq 1$, then as above, we obtain a contradiction. Thus, $w_2=0$, $w_1=m-2$, and $|S_2|=|S|-|T_0|-|S_1|=m+p-2$. If $S_2$ is not product-one free then, as above, we obtain a contradiction. Thus, $S_2$ must be product-one free. Since $S_2$ is contained in exactly one coset of $N$, by Theorem~\ref{inversed(G)}, $G\cong C_2\ltimes C_3$ and $S_2=x\bm\cdot xy \bm\cdot xy^{2}$. Thus, $S=x\bm\cdot xy \bm\cdot xy^{2}\bm\cdot T_0\bm\cdot 1^{[3]}$. Since $\pi(T_0)= \{1\}$, we conclude that either $T_0=x^{[2]}$, or $T_0=(xy)^{[2]}$, or $T_0=(xy^2)^{[2]}$. It is easy to check that $1\in\Pi_6(S)$, yielding a contradiction which proves that $w_1=0$.

We now have $w_2=m-2$, $|S_2|=w_2p+p-1=mp-p-1$, and so $|S_1|=|S|-|T_0|-|S_2|=mp+m+p-3-p-p(m-2)-(p-1)=m+p-2$. By Lemma~\ref{1inpiU}, there exists a subsequence $U\mid S_1$ with $p||U|$ such that either $1\in \pi(U)$ or $1\in \pi(T_0\bm\cdot U)$. Since $1\in\pi(T_0)$, in both cases we have that $1\in \pi(T_0\bm\cdot U)$. Let $t=|U|/p$, then $t\geq 1$ and $1\in \pi(T_0\bm\cdot U \bm\cdot S_{2(v_2+1)}\bm\cdot\ldots\bm\cdot S_{2(v_2+m-1-t)})$. Clearly, $1\in \Pi_{mp}(S)$, yielding a contradiction, and so this subcase is also impossible.
\\

\noindent {\bf Subcase 2.3.} $|I_{\{1\}}|=m-1$.

We have
$$w_1+w_2=|I_{\{1\}}|=m-1.$$
If $\pi(T_0)=\{1\}$, then $1\in \pi(T_0\bm\cdot S_{1(v_1+1)}\bm\cdot\ldots\bm\cdot S_{1(v_1+w_1)}\bm\cdot S_{2(v_2+1)}\bm\cdot\ldots\bm\cdot S_{2(v_2+w_2)})\subseteq \Pi_{mp}(S)$, yielding a contradiction. Thus $$\pi(T_0)\neq \{1\}.$$

We remark that for every factorization of $S$ of the form \eqref{factorization} with $T_0$ fixed, we always have
$$w_1+w_2=|I_{\{1\}}|=m-1.$$
For otherwise, if $|I_{\{1\}}|\leq m-3$, as in Subcase 2.1, we obtain a contradiction. If $|I_{\{1\}}|= m-2$, as in Subcase 2.2, we have $\pi(T_0)= \{1\}$, again yielding a contradiction.

Next, we prove that
$$a_1=0 \mbox{ and } w_2=0.$$
If $w_2\geq 1$, then by replacing $T_0$ with $S_{2(v_2+1)}$, we get a new factorization
$$S=T_0'\bm\cdot S_{11}\bm\cdot\ldots\bm\cdot S_{1v_1}\bm\cdot S_{21}\bm\cdot\ldots\bm\cdot S_{2v_2}\bm\cdot S_{1(v_1+1)}\bm\cdot\ldots\bm\cdot S_{1(v_1+w_1)}\bm\cdot S_{2(v_2+1)}'\bm\cdot\ldots\bm\cdot S_{2(v_2+w_2)}\bm\cdot (W_1\bm\cdot W_2),$$
with $\pi(T_0')=\{1\}$ and $|I_{\{1\}}'|=w_1+w_2-1=m-2$. As in Subcase 2.2, we deduce a contradiction.

If $w_2=0$ and $a_1\neq 0$, then $w_1\geq 1$. By replacing $T_0$ with $S_{1(v_1+1)}$, we get a new factorization
$$S=T_0'\bm\cdot S_{11}\bm\cdot\ldots\bm\cdot S_{1v_1}\bm\cdot S_{21}\bm\cdot\ldots\bm\cdot S_{2v_2}\bm\cdot S_{1(v_1+1)}'\bm\cdot\ldots\bm\cdot S_{1(v_1+w_1)}\bm\cdot S_{2(v_2+1)}\bm\cdot\ldots\bm\cdot S_{2(v_2+w_2)}\bm\cdot (W_1\bm\cdot W_2),$$
with $\pi(T_0')=\{1\}$ and $|I_{\{1\}}'|=w_1+w_2-1=m-2$. As above, we obtain a contradiction. So $a_1=0$ and $w_2=0$.

Then $S_1=S_N$ and $w_1=m-1$. As in the proof of \eqref{w_i>0}, we have
$$S_1=g_1^{[mp-1]}$$
for some $g_1\in S_N$. Since $w_1=m-1>0$, we have $g_1^{p}=1$ and thus $g_1=1$. If $S\bm\cdot S_N^{[-1]}=S\bm\cdot S_1^{[-1]}$ is not product-one free, then let $U\mid S\bm\cdot S_N^{[-1]}$ with maximal length such that $1\in \pi(U)$. Thus, $1\in \pi(U\bm\cdot 1^{[mp-|U|]})$ and so $1\in \Pi_{mp}(S)$, yielding a contradiction. Hence, $S\bm\cdot S_N^{[-1]}$ is product-one free. Note that $|S\bm\cdot S_N^{[-1]}|=mp+m+p-3-(mp-1)=m+p-2$. By Theorem~\ref{inversed(G)}, we have $G\cong C_2\ltimes C_3$ and so $S\bm\cdot S_1^{[-1]}=x\bm\cdot xy \bm\cdot xy^{2}$, thus $S=x\bm\cdot xy \bm\cdot xy^{2}\bm\cdot  1^{[5]}$ as desired.
\end{proof}

\begin{lemma}\label{SNi<p}
Let $S$ be a sequence over $G$ with length $mp+m+p-3$ such that $1\notin \Pi_{mp}(S)$. If $|S_{N_i}|< p$ for all $i\in [1,p-1]$, then
$$\varphi(S\bm\cdot S_N^{[-1]}) \mbox{ is product-one free over } G/N.$$
\end{lemma}
\begin{proof}
Since $|S_{N_i}|< p$ for all $i\in [1,p-1]$, we have $|S\bm\cdot S_N^{[-1]}|=|S_{N_1}\bm\cdot\ldots\bm\cdot S_{N_{p-1}}|\leq (p-1)^2$. Let $n=\mathsf v_1(S)$. We divide the proof into the following two cases. \\

\noindent {\bf Case 1.} $n\geq p+m-2$.\\

Let $S'=S\bm\cdot 1^{[-(m+p-2)]}$ be the subsequence of $S$ obtained by removing $m+p-2$ terms of $1$ from $S$. Since $|S'|=|S|-(m+p-2)=mp-1\geq \mathsf d(G)+1$, $S'$ has a product-one subsequence. Let $T\mid S'$ be a product-one subsequence with maximal length. Then $S'\bm\cdot T^{[-1]}$ is product-one free and thus $|S'\bm\cdot T^{[-1]}|\leq \mathsf d(G)=m+p-2$.

If $|S'\bm\cdot T^{[-1]}|\leq m+p-3$, then $(mp-1)-(m+p-3)\leq |T|\leq mp-1$, and thus $1\leq mp-|T|\leq m+p-2$. Hence, $1\in \pi(T \bm\cdot 1^{[mp-|T|]})$. Since $|T\bm\cdot 1^{[mp-|T|]}|=mp$, we have $1\in \Pi_{mp}(S)$, yielding a contradiction.

Therefore, $|S'\bm\cdot T^{[-1]}|=m+p-2=\mathsf d(G)$. Since $|S_{N_i}|\leq p-1$ for all $i\in [1,p-1]$ and $S'\bm\cdot T^{[-1]}$ is product-one free, by Theorem~\ref{inversed(G)} we have $S'\bm\cdot T^{[-1]}=(x^ay^{b_1})\bm\cdot\ldots\bm\cdot (x^ay^{b_{p-1}})\bm\cdot (y^c)^{[m-1]}$, where $a\in[1,p-1]$, $b_i\in[0,m-1]$ for $i\in[1,p-1]$, $c\in[1,m-1]$, and $\mbox{gcd}(c,m)=1$. Note that $|T|=|S'|-\mathsf d(G)$. By Lemma~\ref{normalsequence}, we have that if $\mathsf v_g(S')\geq 1$, then $\mathsf v_g(S'\bm\cdot T^{[-1]}))\geq 1$ for all $g\in G\setminus \{1\}$. We conclude that $T=T_N$. Otherwise, if there exists a term $g\mid T$ with $g\notin N$, then $g=x^ay^{b_j}$ for some $j\in [1, p-1]$. Thus, $|S_{N_a}|\geq p$, yielding a contradiction.  Therefore, $S\bm\cdot S_N^{[-1]}=(x^ay^{b_1})\bm\cdot\ldots\bm\cdot (x^ay^{b_{p-1}})$, and so clearly, $\varphi(S\bm\cdot S_N^{[-1]})$ is product-one free over $G/N$.\\

\noindent {\bf Case 2.} $n\leq p+m-3$.\\

Assume to the contrary that $\varphi(S\bm\cdot S_N^{[-1]})$ is not product-one free over $G/N$, then we can find a factorization
$$S\bm\cdot S_N^{[-1]}=W_1\bm\cdot \ldots\bm\cdot W_k \bm\cdot W'$$
with $k\geq 1$, where $\varphi(W_i)$ is a minimal product-one subsequence over $G/N$ for $i\in [1,k]$ and $\varphi(W')$ is product-one free over $G/N$. Since $\mathsf d(G/N)=\mathsf d(C_p)=p-1$, we have $|W_k|\leq p$ and $|W'|\leq p-1$. Let $W_k=w_1\bm\cdot \ldots \bm\cdot w_e$ and $w_e=x^iy^j$, where $e=|W_k|$, $1\leq i\leq p-1$ and $0\leq j\leq m-1$. Let $S_N=u_1\bm\cdot\ldots\bm\cdot u_{\ell}$ where $\ell=|S_N|$, $u_t\neq 1$ for $t\in[1,\ell-n]$, and $u_t=1$ for $t\in[\ell-n+1, \ell]$. Then $\ell\geq mp+m+p-3-(p-1)^2$.

We insert each $u_t$ into the product $w_1\ldots w_e$, either before or after $w_e$, where $t\in[1,\ell-n]$. If we put $u_t$ after $w_e$ then it multiplies the product by $u_t$; putting $u_t$ before $w_e$ multiplies the product by $u_t^{r^i}$. Let $A_0=\{w_1\ldots w_e\}$, $A_t=\{u_t, u_t^{r^i}\}$ for $t\in[1,\ell-n]$ and $A_t=\{1\}$ for $t\in[\ell-n+1,\ell]$. Note that $ mp-|W_1\bm\cdot \ldots\bm\cdot W_k|= mp-(|S|-\ell-|W'|)\leq \ell-m+2$. Let $s=mp-|W_1\bm\cdot \ldots\bm\cdot W_k|$, then
$$mp-(p-1)^2\leq s\leq \ell-m+2.$$
Therefore, we have
$$\pi(W_k\bm\cdot u_{i_1}\bm\cdot\ldots\bm\cdot u_{i_s})\supseteq A_0A_{i_1}\ldots A_{i_s}$$
for every $s$-subset $\{i_1,\ldots,i_s\}\subseteq [1,\ell]$. Let $\mathbf{A}=(A_1,\ldots, A_{\ell})$ and $M=\mbox{stab}(\Pi^{s}(\mathbf{A}))$. Recall that $I_M$ is the subset of $[1,\ell]$ such that $j\in I_M$ if and only if $A_j\subseteq M$. As before, if $|I_M|\geq mp+|M|-1$, then there exists a product-one subsequence of $S_M$ with length $mp$, yielding a contradiction. Therefore, $|I_M|\leq \min\{\ell, mp+|M|-2\}$.

We first prove $|\Pi^{s}(\mathbf{A})|\geq m$. As in the proof of Lemma~\ref{1inpiU}, we may always assume that $M\lneq N$. Let $$\mu=|\{Q\in N/M: |V_Q|\geq s+1\}|.$$ If $\mu=0$, then $|I_M|\leq s$. By Lemma~\ref{mu}~(i),
\begin{align*}
|\Pi^{s}(\mathbf{A})|&\geq |M|(1-s+2\ell-|I_M|)\\
&\geq |M|(1+2(\ell-s))\\
&\geq |M|(1+2(m-2)) \ \ \ \         (\mbox{as } s\leq \ell-m+2)\\
&\geq m.
\end{align*}

If $\mu\geq 2$, then by Lemma~\ref{mu}~(ii),
\begin{align*}
|\Pi^{s}(\mathbf{A})|\geq (s+1)|M| \geq m \ \ \ \ \ \ \ \         (\mbox{as } s\geq mp-(p-1)^2\geq m).
\end{align*}

If $\mu=1$ and $R\neq M$, then by Lemma~\ref{mu}~(iv),
\begin{align*}
|\Pi^{s}(\mathbf{A})|\geq |M|(1+\ell)\geq m.
\end{align*}

If $\mu=1$ and $R=M$, then $V_R=I_M$. Note that $|I_{\{1\}}|\leq m+p-3\leq mp-(p-1)^2\leq s$. Thus, $M=R\neq \{1\}$ and so $|V_R|=|I_M|\leq\min\{\ell, mp+|M|-2\}=mp+|M|-2$. Therefore, by Lemma~\ref{mu}~(iii),
\begin{align*}
|\Pi^{s}(\mathbf{A})|
&\geq |M|(1+2(\ell-|I_M|))\\
&\geq |M|(2(m-|M|-(p-1)(p-2))+1) \ \ \ \ \ \ (\mbox{as } |I_M|\leq mp+|M|-2)\\
&= (m-|M|-(p-1)(p-2))(2|M|-1)-(p-1)(p-2)+m\\
&\geq m \ \ \ \ \ \ \ \ \ \ (\mbox{as } |M|\geq p+1 \mbox{ and } m\geq (p+1)|M|).
\end{align*}
In all the cases, we have shown $|\Pi^{s}(\mathbf{A})|\geq m$, whence $A_0\Pi^{s}(\mathbf{A})=N$. Thus, we can find a subsequence $U\mid u_1\bm\cdot\ldots\bm\cdot u_{\ell}$ with length $|U|=s$ such that $1\in \pi(W_1\bm\cdot \ldots\bm\cdot W_k \bm\cdot U)$. Since $|W_1\bm\cdot \ldots\bm\cdot W_k \bm\cdot U|=mp$, we have $1\in \Pi_{mp}(S)$, yielding a contradiction. This completes the proof of the lemma.
\end{proof}

We are now in position to prove our second main result.\\

\noindent {\bf Proof of Theorem \ref{inverseE(G)}.}\\

$(a)\Rightarrow (b)$. Let $S$ be a sequence over $G$ with length $mp+m+p-3$ such that $1\notin \Pi_{mp}(S)$.

(i) Since $G\not\cong C_2\ltimes C_3$, by Lemmas~\ref{SNigeqp} and \ref{SNi<p}, we have $\varphi(S\bm\cdot S_N^{[-1]})$ is product-one free over $G/N$, so $|S\bm\cdot S_N^{[-1]}|\leq p-1$. If $|S_N|\geq m(p-1)+2m-1=mp+m-1$, by Lemma~\ref{knzerosumfree}~(i), there exist $p$ disjoint product-one subsequences $T_1, \ldots, T_p$ of $S_N$ with length $|T_j|=m$ for all $j\in [1,p]$. Therefore, $T_1\bm\cdot\ldots\bm\cdot T_p$ is a product-one subsequence of $S$ with length $mp$, yielding a contradiction. Hence $|S_N|\leq mp+m-2$. Since $p-1\geq |S\bm\cdot S_N^{[-1]}|=|S|-|S_N|\geq p-1$, we have $|S_N|=mp+m-2$ and $|S\bm\cdot S_N^{[-1]}|= p-1$. Therefore, by Lemma~\ref{n-1}, $|\supp(\varphi(S\bm\cdot S_N^{[-1]}))|=1$ and thus $S\bm\cdot S_N^{[-1]}$ is contained in a single coset of $N$, so $S\bm\cdot S_N^{[-1]}=(x^ay^{b_1})\bm\cdot\ldots\bm\cdot (x^ay^{b_{p-1}})$ where $a\in [1,p-1]$ and $b_i\in [0,m-1]$ for every $i\in[1,p-1]$.  Since $1\notin\Pi_{mp}(S_N)$ and $|S_N|=mp+m-2$, by Lemma~\ref{knzerosumfree}~(ii), $S_N=(y^{c_1})^{[k_1m-1]}\bm\cdot (y^{c_2})^{[k_2m-1]}$ where $c_1,c_2\in[0,m-1]$, $\mbox{gcd}(c_1-c_2,m)=1$ and $k_1+k_2=p+1$. Putting this together, we obtain
\begin{align}\label{Gextremal}
S=(S\bm\cdot S_N^{[-1]})\bm\cdot S_N=(x^ay^{b_1})\bm\cdot\ldots\bm\cdot (x^ay^{b_{p-1}})\bm\cdot (y^{c_1})^{[k_1m-1]}\bm\cdot (y^{c_2})^{[k_2m-1]}
\end{align}
as desired.

(ii) Since $G\cong C_2\ltimes C_3$, we have $p=2$ and $m=3$. If $|S_{N_1}|\geq p=2$, then by Lemma~\ref{SNigeqp}, $S=x\bm\cdot xy \bm\cdot xy^{2}\bm\cdot  1^{[5]}$, as desired. If $|S_{N_1}|< p=2$, then by Lemma~\ref{SNi<p},
$\varphi(S\bm\cdot S_N^{[-1]})$ is product-one free over  $G/N$. As in (i), we have $S$ is of the form described in \eqref{Gextremal} and thus $S=(xy^{b})\bm\cdot (y^{c_1})^{[5]}\bm\cdot (y^{c_2})^{[2]}$, where $b, c_1,c_2\in [0,2]$ with $\mbox{gcd}(c_1-c_2,3)=1$.

$(b)\Rightarrow (a)$. (i) Let $S'$ be any subsequence of $S$ with length $|S'|=mp$. To show that $S$ is $|G|$-product-one free (i.e., $1\notin \Pi_{mp}(S)$), it suffices to show that $1\notin \pi(S')$. Write $S'=S_1'\bm\cdot S_2'$ such that $S_1'\mid (x^ay^{b_1})\bm\cdot\ldots\bm\cdot (x^ay^{b_{p-1}})$ and $S_2'\mid (y^{c_1})^{[k_1m-1]}\bm\cdot (y^{c_2})^{[k_2m-1]}$. If $S_1'\neq \emptyset$, then $\varphi(S')$ is not a product-one sequence over $G/N$ since $a\in[1,p-1]$ and $1\neq |S_1'|\leq p-1$. Thus, $\pi(S')\cap N=\emptyset$ and so $1\notin\pi(S')$. If $S_1'=\emptyset$, then $S'\mid (y^{c_1})^{[k_1m-1]}\bm\cdot (y^{c_2})^{[k_2m-1]}$. Since $|(y^{c_1})^{[k_1m-1]}\bm\cdot (y^{c_2})^{[k_2m-1]}|=mp+m-2$ and $\mbox{gcd}(c_1-c_2,m)=1$, by Lemma~\ref{knzerosumfree}~(ii), $1\notin \pi(S')$, as desired.

(ii) If $G\cong C_2\ltimes C_3$ and $S=(xy^{b})\bm\cdot (y^{c_1})^{[5]}\bm\cdot (y^{c_2})^{[2]}$, then as in (i), we have $1\notin\Pi_{6}(S)$. If $G\cong C_2\ltimes C_3$ and $S=x\bm\cdot xy \bm\cdot xy^{2}\bm\cdot 1^{[5]}$ then, by Theorem~\ref{inversed(G)}, we have $x\bm\cdot xy \bm\cdot xy^{2}$ is product-one free, implying $1\notin\Pi_{6}(S)$. \qed
\\

\noindent {\bf Acknowledgements}. This work was carried out during a visit by the first author to Brock University
as an international visiting scholar. He would like to sincerely thank the host institution for its hospitality
and for providing an excellent atmosphere for research. This work was supported in part by the National Science
Foundation of China (Grant No. 11701256, 11871258), the Youth Backbone Teacher Foundation of Henan's University (Grant No. 2019GGJS196), the China Scholarship Council (Grant No. 201908410132), and also in part by a discovery grant from the Natural Sciences and Engineering Research Council of Canada (Grant No. RGPIN 2017-03903).

\end{document}